\documentclass[11pt]{amsart}
\usepackage{bm}
\usepackage{amsmath}
\usepackage{amsthm,amssymb,mathrsfs}
\usepackage{epic,eepic}
\usepackage{paralist,enumerate}
\usepackage[all]{xy}
\def\Log{{\rm Log}}
\def\Supp{{\rm supp}}

\usepackage{tikz-cd}
\usepackage{xcolor,graphicx}
\usepackage[mathscr]{euscript}
\usepackage{circuitikz}
\ctikzset{bipoles/resistor/height=0.15}
\ctikzset{bipoles/resistor/width=0.4}

\definecolor{cadmiumgreen}{rgb}{0.0, 0.42, 0.24}
\usepackage[
colorlinks, citecolor=cadmiumgreen,
pdfauthor={Farhad Babaee},
pdftitle={ },
pdfstartview ={FitV},
]{hyperref}

\usepackage[
alphabetic,
msc-links,
nobysame,
lite,
]{amsrefs}

\usepackage{tikz, float}
\usepackage{tikz-cd}
\usetikzlibrary {positioning}

\usetikzlibrary{calc,decorations.markings}
\usetikzlibrary{shapes,snakes}

\usepackage[left=3.5cm,top=3.5cm,right=3.5cm]{geometry}
\DeclareSymbolFont{extraup}{U}{zavm}{m}{n}
\DeclareMathSymbol{\varheart}{\mathalpha}{extraup}{86}
\DeclareMathSymbol{\vardiamond}{\mathalpha}{extraup}{87}

\theoremstyle{definition}
\newtheorem{mtheorem}{Theorem}
\newtheorem{mconjecture}[mtheorem]{Conjecture}

\newtheorem{theorem}{Theorem}[section]
\newtheorem{definition}[theorem]{Definition}
\newtheorem{lemma}[theorem]{Lemma}
\newtheorem{proposition}[theorem]{Proposition}
\newtheorem{corollary}[theorem]{Corollary}
\newtheorem*{theorem*}{Theorem}

\usepackage{comment}
\newtheorem{remark}[theorem]{Remark}
\newtheorem{example}[theorem]{Example}

\newcommand{\RR}{\mathbb{R}}
\newcommand{\R}{\mathbb{R}}

\newcommand{\ZZ}{\mathbb{Z}}
\newcommand{\Z}{\mathbb{Z}}

\newcommand{\CC}{\mathbb{C}}
\newcommand{\C}{\mathbb{C}}

\newcommand{\T}{\mathbb{T}}

\usepackage{bbm}

\newcommand{\trop}{\mathrm{trop}}

\newcommand{\sD}{\mathscr{D}}

\newcommand{\cO}{\mathcal{O}}

\newcommand{\cC}{\mathcal{C}}
\newcommand{\sT}{\mathscr{T}}

\newcommand{\sS}{\mathscr{S}}
\newcommand{\sZ}{\mathscr{Z}}

\newcommand{\Spec}{\mathrm{Spec}}
\newcommand{\Hom}{\mathrm{Hom}}
\newcommand{\Trop}{\mathrm{trop}}

\usepackage{hyperref}
\hypersetup{
  colorlinks=true,
  linkcolor=blue,
  filecolor=magenta,
  urlcolor=cyan,
  pdftitle={Overleaf Example},
  pdfpagemode=FullScreen,
  }

\renewcommand{\P}{\mathbb P}

\usetikzlibrary{graphs,quotes}

 \newcommand{\supp}[1]{{\mathrm{supp}(#1)}}

\usepackage{scalerel}

\newcommand{\rmin}{\mathbf{r}}
\newcommand{\rminm}{\mathbf{r}}
\newcommand{\rmax}{\mathbf{R}}

\begin{document}
% \begin{center}
% \emph{To the memory of Nessim Sibony}
% \end{center}
\title{Dynamical tropicalisation}
\author{Farhad Babaee}

\address{School of Mathematics, University of Bristol}
\email{farhad.babaee@bristol.ac.uk}
\date{}
\begin{abstract}
We analyse the dynamics of the pullback of the map $z \longmapsto z^m$ on the complex tori and toric varieties. We will observe that tropical objects naturally appear in the limit, and review several theorems in tropical geometry.
\end{abstract}
\maketitle
\setcounter{tocdepth}{1}
\tableofcontents
\section{Introduction}
In this article, we employ several key ideas from tropical geometry to analyse the dynamics of the pullback of the map
\begin{align*}
  \Phi_m: (\CC^*)^n &\longrightarrow (\CC^*)^n \\
  (z_1, \dots, z_n) &\longmapsto (z_1^m, \dots , z_n^m),
\end{align*}
as the positive integer $m\to \infty.$ This analysis, in turn, lets us review certain aspects of tropical geometry from a dynamical standpoint. To start with, let us examine the case $n=1.$ Observe that for \emph{any} $z\in \C^*,$ the family of sets $\big\{\Phi_m^{-1}(z)\big\}$ \emph{converges} towards a uniform distribution of points on the unit circle $S^1$ as $m\to \infty$. Formally, this observation can be formulated as the following \emph{equidistribution theorem}: for any $z \in \C^*,$
$$
\frac{1}{m}\Phi_m^{*}(\delta_z):=\frac{1}{m}\sum_{ \Phi_m(a)=z} \delta_{a} \longrightarrow \mu(S^1),\quad \text{as $m\to \infty$},
$$
where $\delta_z$ is the Dirac measure at $z$, $\mu(S^1)$ is the Haar measure on $S^1,$ and the limit is in the weak sense of measures. Brolin's remarkable theorem \cite{Brolin} is a generalisation of this observation: for \emph{any} monic algebraic map of $f: \C \longrightarrow \C$ of degree $d\geq 2,$ there exists a (harmonic) measure $\mu_f$, such that for any \emph{generic} point $z\in \C$,
  $$
  \frac{1}{d^k}(f^{k})^*(\delta_{z}) \longrightarrow \mu_f, \quad \text{as $k\to \infty$},
  $$
  where $f^k:= f\circ \dots \circ f$ is the $k$-fold composition. This result was extended to $\P^1$ for polynomials in \cite{Ljubich} and for rational maps in \cite{FLM:Brolin}. Here, the genericity of a point in $\C$ and $\P^1$  means that $z$ is outside an invariant set of cardinality less than or equal to $1$ and $2$, respectively.
  \vskip 2mm
  In complex dynamics, one naturally seeks the generalisation of Brolin's theorem in higher dimensions and codimensions, and the question can be suitably formulated in the language of \emph{currents}, where one explores weak limits of \emph{pullback} of currents after a correct normalisation. Informally, currents on a complex smooth manifold $X$ are continuous functionals acting on the space of smooth forms with compact support and of the appropriate (bi-)degree. For instance, an algebraic subvariety $Z\subseteq X$ of dimension $p$ defines an \emph{integration current} $[Z]$, which is of bidimension $(p,p)$ and acts on the smooth forms with compact support of bidegree $(p,p)$ by integration:
  $$
  \langle [Z] , \varphi \rangle := \int_{Z_{\text{reg}}} \varphi.
  $$
  Generally, given a smooth algebraic variety $X$, $f:X \longrightarrow X$ a holomorphic endomorphism of algebraic degree $d,$  and $Z \subseteq X$ an algebraic subvariety of dimension $p,$ one is interested in analysing (the existence of) the weak limit of
  $$
  \frac{1}{d^{(n-p)k}} (f^k)^*[Z]:= \frac{1}{d^{(n-p)k}} [(f^k)^{-1}(Z)],
  $$
  as $k \to \infty.$ We recall basic definitions in the theory of currents in Section~\ref{sec:prem_currents}.
  \vskip 2mm
  On the tropical geometry side, algebraic subvarieties of the torus are degenerated to obtain their \emph{tropicalisation}. For instance, a tropicalisation can be obtained by the \emph{logarithm map}
$$
\Log: (\CC^*)^n \longrightarrow \R^n, \quad (z_1, \dots, z_n) \longmapsto (-\log|z_1|, \dots , -\log|z_n|).
$$
Given an algebraic subvariety $Z\subseteq (\C^*)^n$, the set $\Log(Z)$ is called the \textit{amoeba of} $Z.$ By Bergman's theorem \cite{Bergman} there exists a close subset of $\R^n$ such that
$$
\frac{1}{\log|t|}\Log(Z) \longrightarrow \cC,\quad \text{as $|t| \to \infty$},
$$
where the limit is in the Hausdorff metric in the compact sets of $\RR^n,$ and $t$ is a complex parameter. It is further shown in tropical geometry that $\cC$ can be equipped with a structure of a \emph{tropical cycle}; see Definition~\ref{BalancingCondition} and Section~\ref{sec:trop-toric}. With such an induced structure, we set $\cC$ to be the tropicalisation of $Z$, denoted by $\trop(Z)$.
\vskip 2mm
Note that for any point $z\in (\C^*)^n,$
$$\frac{1}{\log|t|}\Log(\{z\}) \longrightarrow \{ 0\} ,\quad \text{as $|t| \to \infty$},$$
which corresponds to the above-mentioned equidistribution theorem.
%We note that we have the homothecy relation of the sets $\Log(\Phi_m^{-1}(Z)) = %\frac{1}{m} \Log(Z),$
%which implies that the family sets $\Phi_m^{-1}(Z)$ converge in Hausdorff metric to
%$\Log^{-1}(\trop(Z)).$
More generally, our first theorem shows how naturally tropical objects emerge as geometric objects in holomorphic dynamics.
\begin{mtheorem}\label{thm:main_intro}
Let $Z\subseteq (\CC^*)^n$ be an irreducible subvariety of dimension $p,$ then
$$
\frac{1}{m^{n-p}}\Phi_m^*[Z] \longrightarrow \mathscr{T}_{\trop(Z)}, \quad \text{as $m\to \infty$},
$$
where $\mathscr{T}_{\trop(Z)}$ is the \emph{complex tropical current} associated to $\trop(Z).$
\end{mtheorem}
 A complex tropical current $\sT_{\cC}$ is a closed current of bidimension $(p,p),$ associated to the tropical cycle $\cC\subseteq \RR^n$ of pure dimension $p$, and with support $\Log^{-1}(|\cC|).$ Here $|\cC|$ denotes the support of $\cC$ which is obtained by forgetting its polyhedral structure. Complex tropical currents were introduced in \cite{Babaee}, and we recall their definition in Section~\ref{sec:trop_currents}.
\vskip 2mm
Noting that the pullback can be extended to more general currents, we introduce the following  definition.
\begin{definition}
Let $\sT$ be a closed positive current of bidimension $(p,p)$ on $(\C^*)^n$. We define the \emph{dynamical tropicalisation} of $\sT$ as 
$$
\lim_{m\to \infty} \frac{1}{m^{n-p}} \Phi^*_m (\sT),
$$
when the limit exists.
\end{definition}
 Theorem~\ref{thm:main_intro}, therefore, states that the dynamical tropicalisation of an integration current along a subvariety of the torus yields the tropical current associated to the tropicalisation of that subvariety. Even though the preceding theorem is stated only in $(\C^*)^n,$ we require a passage to \emph{toric varieties} to provide a proof. Recall that a toric variety is an irreducible variety $X$ such that
  \begin{enumerate}
  \item [(a)] $(\C^*)^n$ is a Zariski open subset of $X,$ and
  \item [(b)] the action of $(\C^*)^n$ on itself extends to an action of $(\C^*)^n$ on $X.$
  \end{enumerate}
The action of $(\C^*)^n$ on $X$ partitions $X$ into \emph{orbits},  and we will observe in Section~\ref{sec:trop-toric} that the continuous extension of the endomorphism $\Phi_m$ from $(\C^*)^n$ to $X$ gives rise to an equidistribution theorem of points within each orbit; see Proposition~\ref{prop:equi_orbit}.
%Recall that, by definition, a \emph{toric variety} $X$ is an irreducible variety such that %\begin{itemize}
% \item [(a)] $(\C^*)^n$ is a Zariski open subset of $X$, and
% \item [(b)] the action of $(\C^*)^n$ on itself extends to an action of $(\C^*)^n$ on $X$.
%\end{itemize}
%As a result, $\Phi_m$ can be easily extended as an \emph{toric endomorphism} on any toric variety, and
\vskip 2mm
The above-mentioned tropicalisation corresponds to the tropicalisation with respect to the \emph{trivial valuation}. In essence, this tropicalisation captures the (exponential) directions where a subvariety $Z\subseteq (\C^*)^n$ approaches infinity, and there always exists a toric variety $X$ that can contain those directions at infinity so that
\begin{itemize}
    \item [(a)]the closure of $Z$ in $X$ is compact, and
    \item [(b)]the the boundary divisors $X\setminus (\C^*)^n$ intersect the closure of $Z$ in $X$ properly.
\end{itemize}
This closure is called the \emph{tropical compactification} of $Z$ in $X;$ see Definition~\ref{def:trop_comp}. In Subection~\ref{sec:toric_tori}, we extend Theorem~\ref{thm:main_intro} to toric varieties: 
\begin{mtheorem}\label{thm:intro_toric}
Let $Z\subseteq (\CC^*)^n$ be an irreducible subvariety of dimension $p,$ and $\bar{Z}$  be the {tropical compactification} of $Z$ in the smooth projective toric variety $X$. Then,
$$
\frac{1}{m^{n-p}}\Phi_m^*[\bar{Z}] \longrightarrow \overline{\mathscr{T}}_{\trop(Z)}, \quad \text{as $m\to \infty$},
$$
where $\Phi_m: X\longrightarrow X$ is the continuous extension of $\Phi_m: (\C^*)^n \longrightarrow (\C^*)^n,$ and $\overline{\mathscr{T}}_{\trop(Z)}$ is the extension by zero of $\mathscr{T}_{\trop(Z)}$ to $X.$
\end{mtheorem}
Theorem~\ref{thm:intro_toric} provides a vivid picture of dynamical tropicalisation which is illustrated in Figure~\ref{fig:first_fig}. 
In Subsection~\ref{sec:Kapranov}, we observe that the dynamical tropicalisation of Poincar\'e--Lelong equation, see Theorem~\ref{thm:LP}, yields the following dynamical version of Kapranov's theorem, which informally states that the tropicalisation of an algebraic hypersurface $V(f)$ coincides with the tropical variety of the tropicalisation of the polynomial $f$:

\begin{mtheorem}\label{thm:intro_Kapranov}
Let $V\subseteq (\C^*)^n,$ be an irreducible algebraic hypersurface, given as the variety of the polynomial $f(z)= \sum_{\alpha} c_{\alpha} z^{\alpha}\in \C[z]$. We have that
$$m^{-1} \Phi_m^*[V(f)] \longrightarrow \sT_{V_{\trop}(\mathfrak{q})},$$
where $\mathfrak{q}= \max_{\alpha}\{ \langle -\alpha ,~\cdot~ \rangle \}:\R^n \longrightarrow \R,$ is the \emph{tropicalisation of} $f$ and ${V_{\trop}(\mathfrak{q})}$ is the \emph{tropical variety associated} to $\mathfrak{q}$; see Definition \ref{def:tropical_poly}.
\end{mtheorem}

% \begin{example}\label{ex:dim2}
% $\Phi_2:\P^2 \longrightarrow \P^2,$ $[z_0^2:z_1^2:z_2^2]$ analyse $z+w + 1 =0, z^2+ w^3 +1 =0.$
 % \end{example}
\vskip 2mm
At the beginning of the introduction, we mentioned Brolin's theorem and some generalisations. Let us now recall an equidistribution conjecture of Dinh and Sibony and some of the known special cases. This will allow for viewing Theorems~\ref{thm:main_intro} and \ref{thm:intro_toric} within the larger context of complex dynamics. For an overview of the current trends in Complex Dynamics in Higher Codimensions see Dinh's ICM 2018 survey \cite{Dinh-ICM}, and Dujardin's ICM 2022 survey \cite{Dujardin-ICM} on Geometric Methods in Holomorphic Dynamics. \vskip 2mm
Let $\mathscr{H}_d(\P^n)$ denote the set of holomorphic endomorphisms of degree $d$ on $\P^n,$ and assume that $d\geq 2.$  The \emph{Green current} of $f \in \mathscr{H}_d(\P^n),$ is defined by 
  $$\sT_{f}:= \lim_{k\to \infty} \frac{1}{d^{k}} (f^k)^*(\omega),$$ where $\omega$ is the Fubini--Study form  cohomologous to a hyperplane in $\P^n.$  
\begin{mconjecture}[Dinh--Sibony \cites{Dinh-Sibony-distribution,Dinh-ICM}]
  For any $f \in \mathscr{H}_d(\P^n)$, any integer $p$ with $1 \leq p\leq n-1,$ and \emph{generic} subvariety $Z\subseteq\P^n$ of dimension $p$, we have the weak convergence of positive closed currents
  $$
  \frac{1}{\deg Z}\frac{1}{d^{(n-p)k}} (f^k)^*[Z]\longrightarrow \sT_f^{n-p}, \quad \text{as $k\to \infty$},
  $$
  where $\sT_f^{n-p} = \sT_f \wedge \dots \wedge \sT_f$, is the $(n-p)$-fold \emph{wedge product} of the Green current of $f,$ and in particular, the limit only depends on $f.$ 
\end{mconjecture}
Here are the special known cases of the above conjecture. 
  \begin{itemize}
  \item [(a)] The case $p=0$ was shown by Forn\ae ss and Sibony \cite{Fornaess--Sibony-higherdim}, Briend and Duval \cite{Briend--Duval}, Dinh and Sibony \cite{Dinh--Sibony-dyn-allure}.
  \item [(b)] The case $p=n-1$ was proved by Forn\ae ss and Sibony for generic maps \cite{Fornaess--Sibony-ComplexDynII}, and by Favre and Jonsson for any map in dimension $2$ in \cite{Favre--Jonsson-Brolin}.
  \item [(c)] The case $p=n-1$ was proved  by Dinh and Sibony in \cite{Dinh--Sibony-EquiGreen}.
  \item [(d)] In any dimension for generic maps $f \in \mathscr{H}_d(\P^n)$ by Dinh and Sibony in \cite{Dinh-Sibony:superpot}.
\end{itemize}
The above conjecture also predicts an exponential rate of convergence which is studied in several works, see for instance \cite{Dinh-Sibony:superpot} and \cite{Taflin}. There are also non-archimedean equidistribution theorems and we refer the reader to works Rivera-Letelier \cite{Rivera-Letelier}, Baker and Rumley \cite{Baker--Rumely}, and to Johnsson's comprehensive survey of results and references \cite{Jonsson-survey}.
%   \item [(d)] Dinh--Sibony Theorem \cite{Dinh-Sibony:superpot}*{Theorem 1.0.2}. There exists an \emph{explicit} Zariski open subset $\mathscr{H}^*_d(\P^n) \subseteq \mathscr{H}_d(\P^n)$ such that, if $f\in \mathscr{H}^*_d(\P^n),$ then
%   $$
%   \frac{1}{\deg Z}\frac{1}{d^{(n-p)k}} (f^k)^*[Z]\longrightarrow \sT_f^{n-p}, \quad \text{as $k\to \infty,$}
%   $$
%   for any subvariety $Z$ of dimension $p.$
%   \vskip 2mm
%   \noindent
%   %see Subsection~\ref{sec:no-contradiction}. 
%   \vskip 2mm
%   \item [(b)] Dinh--Sibony Conjecture \cites{Dinh-Sibony-distribution,Dinh-ICM}.
%   For any $f \in \mathscr{H}_d(\P^n)$, any integer $p$ with $2 \leq p\leq n-1,$ and \emph{generic} subvariety $Z\subseteq\P^n$ of dimension $p$, we have
%   $$
%   \frac{1}{\deg Z}\frac{1}{d^{(n-p)k}} (f^k)^*[Z]\longrightarrow \sT_f^{n-p}, \quad \text{as $k\to \infty$},
%   $$
%   where $\sT_f$ is the Green current of $f.$
%   \end{itemize}
\vskip 2mm
We remark further that
  \begin{itemize}
\item [(a)] The map $\Phi_d$ is not a generic map in $\mathscr{H}_d(\P^n)$ in the sense of Item (d) above.  Furthermore, comparing our result to the above conjecture implies that generic subvarieties of the projective space, up to the degree, have the same tropicalisation. This is not a contradiction, in fact, it is the main result of \cite{Romer-Schmitz} for which we provide two other explanations in Subsection~\ref{sec:no-contradiction} and will observe that the \emph{Julia sets} of various orders of $\Phi_m$ relate to the support of \emph{Bergman fan of uniform matroids.} 
  
  \item [(b)] The map $\Phi_m$ is indeed versatile. This map was already employed in \cite{Rash} to find the \emph{logarithmic indicators} of \emph{plurisubharmonic functions} with logarithmic growth. Moreover, Fujino in \cite{Fujino} used  $\Phi_{m}$  to prove certain vanishing theorems on toric varieties; 
  
  \item [(c)] In Subsection~\ref{sec:Kapranov}, we see that $\Phi_m^*$, after a normalisation, applied to Lelong--Poincar\'e equation~\ref{thm:LP},  yields a version of Kapranov's theorem, where an analogue of Maslov dequantisation naturally appears; see Section~\ref{sec:trop_algebra}. Finally, it is alluded in the proof of Theorem~\ref{thm:main_conv} that $\Phi_m^*,$ after a normalisation preserves the Chow cohomology classes. We will gather further cohomological implications of the latter fact in a subsequent article \cite{Babaee-Gualdi}. 
  
%  \item We finally hope that further development intersection theory of tropical currents in relation to dynamical tropicalisation will be helpful in translating the insights from complex dynamics in tropical geometry and . 
  
%   \item On the one hand, in complex dynamics it is useful to have a precise analysis of dynamics of a holomophic map on currents as Theorems \ref{thm:main_intro} and \ref{thm:intro_toric} provide, and moreover, one can develop the . Therefore, we hope that the objects in this work, together with the powerful computational techniques in tropical geometry, can provide a suitable test-space for questions in complex dynamics as in \cite{BH}. On the other hand, Fujino in \cite{Fujino} anticipated that a systematic study of toric varieties via the map $\Phi_m,$ must be possible, and we speculate that the setting of tropical currents on toric varieties can provide a suitable setting for such a study. 
   \end{itemize}
 
  In Sections~\ref{sec:prem_currents}, \ref{sec:trop_algebra}, \ref{sec:trop_currents} and \ref{sec:trop-toric} of this manuscript, we provide the background required in proving our main theorems. The entirety of Section~\ref{sec:main_proof} is devoted to the proof of Theorem~\ref{thm:main_intro}. In Section~\ref{sec:apps} we prove Theorems~\ref{thm:intro_toric} and \ref{thm:intro_Kapranov} and discuss the implication of the Dinh--Sibony Conjecture in the case of the map $\Phi_m$.

\section{Preliminaries of the Theory of Currents}\label{sec:prem_currents}
The content of this section is extracted from Demailly's book \cite{DemaillyBook1}, which has always been generously publicly available.  
\vskip 2mm
Let $X$ be a complex manifold of dimension $n$. For a non-negative integer $k$, we consider the space of smooth complex differential forms of degree $k$ with compact support, denoted by $\mathscr{D}^{k}(X),$ endowed with the inductive limit topology. The \textit{topological dual} to $\mathscr{D}^{k}(X),$ constitutes the space of currents of dimension $k$,
\[
\mathscr{D}'_{k}(X):=\big(\mathscr{D}^{k}(X)\big)',
\]
that is, the space of all continuous linear functionals on $\mathscr{D}^{k}(X)$. Hence, a current $\sT \in \mathscr{D}'_{k}(X)$ acts on any form $\varphi \in \mathscr{D}^{k}(X)$ and yields $\langle \sT, \varphi \rangle \in \C.$
\vskip 2mm
\noindent
We note that when $k=0,$ the space of $k$-currents is just the space of distributions as in the theory of partial differential equations. The \emph{support} of $\sT$, denoted by $\supp{\sT}$, is the smallest closed subset $S\subseteq X,$ such that $\sT$ vanishes on its complement. A $k$-dimensional current $\mathscr{T}$ is \emph{the weak limit} of a sequence of $k$-dimensional currents $\mathscr{T}_i$ if
\[
\lim_{i \to \infty} \langle \mathscr{T}_i, \varphi \rangle = \langle \mathscr{T}, \varphi \rangle, \ \ \text{for all $\varphi \in \mathscr{D}^{k}(X)$}.
\]
We say that $\mathscr{T}$ is a \emph{cluster value} of a sequence $\mathscr{T}_i$ if $\sT$ is the weak limit of a subsequence of $\sT_i.$ 

\begin{remark}\label{rem:supp_haus}
Note that if $\sT$ is the weak limit  of the sequence $\mathscr{T}_i,$ and the sets $\supp{\sT_i}$ converge in the Hausdorff metric to the set $S,$ then $\supp{\sT} \subseteq S.$ To see this, assume that $z\notin S,$ then for any large $i\gg 0,$ $\sT_i$'s vanish in a small neighbourhood of $z,$ and therefore $z\notin \supp{\sT}$. As an example for the strict inclusion  $\supp{\sT} \subsetneq S,$ for  any $z\in X,$ consider the weak limit $\sT_i:=  {i}^{-1}\delta_{z}\longrightarrow 0.$ We have $S= \supp{\sT_i}= \{z\},$ and $\supp{\sT}=\varnothing.$
\end{remark}
The exterior derivative of a $k$-dimensional current $\mathscr{T}$ is the $(k-1)$-dimensional current $d\mathscr{T}$ defined by
\[
\langle d\mathscr{T},\varphi \rangle = (-1)^{k+1} \langle \mathscr{T}, d\varphi \rangle, \quad \varphi \in \mathscr{D}^{k-1}(X).
\]
The current $\mathscr{T}$ is called \emph{closed} if $d\sT = 0$. The duality of currents and forms with compact support induces the following decompositions for the bidegree and bidimension
\[
\mathscr{D}^k(X)=\bigoplus_{p+q=k} \mathscr{D}^{p,q}(X), \quad \mathscr{D}'_k(X)=\bigoplus_{p+q=k} \mathscr{D}'_{p,q}(X).
\]
The space of smooth differential forms of bidegree $(p,p)$ contains the cone of \textit{positive differential forms}. By definition, a smooth differential $(p,p)$-form $\varphi$ is \emph{positive} if
\[
\text{$\varphi(x){\upharpoonright}_S $ is a nonnegative volume form for all complex $p$-planes $S \subseteq T_x X$ and $x \in X$}.
\]
A current $\mathscr{T}$ of bidimension $(p,p)$ is \emph{positive} if
\[
\langle \mathscr{T}, \varphi \rangle \ge 0 \ \ \text{for every positive differential $(p,p)$-form $\varphi$ on $X$}.
\]
\vskip 2mm
An important class of positive currents on $X$ is obtained by integrating along the complex analytic subsets of $X,$ giving rise to \textit{integration currents.} More precisely, if $Z$ is a $p$-dimensional complex analytic subset of $X$, then by Lelong's theorem \cite[Theorem III.2.7]{DemaillyBook1} the {integration current} along $Z$ is the well-defined $(p,p)$-dimensional current given by 
\[
\big\langle [Z], \varphi \big\rangle = \int_{Z_{\text{reg}}} \varphi, \quad \varphi \in \mathscr{D}^{p,p}(X).
\]
\begin{remark}
For simplicity, we only mention the notion of positive currents. To review \textit{weakly positive, positive} and \textit{strongly positive} forms and currents see \cite{DemaillyBook1}.
\end{remark}

Let us assume further that $X$ is a K\"ahler manifold, with the K\"ahler form $\omega$.
One can check that all positive currents have signed measure coefficients. Accordingly, for a positive $(p,p)$-dimensional current $\sT$ the \emph{trace measure} of $\sT$ is defined as $\sT\wedge \omega^{p}.$ Note that the trace measure of any positive current is a positive measure. The local mass of $\sT$ on a Borel set $K\subseteq X,$ with respect to $\omega,$ is given by
$$
\lVert \sT \rVert_{K} = \int_X \mathbbm{1}_K \sT \wedge \omega^p.
$$
%The total mass of $\sT$ is just $\lVert \sT \rVert_{X}.$
%Note that we have the weakly convergent sequence of positive crruents $\sT_j \longrightarrow \sT$, if and only if, the weak convergence of measures $\sT_j \wedge \omega^p \longrightarrow \sT\wedge \omega^p$, is verified.
\vskip 2mm
The differentials $\partial, \bar{\partial}$ and the de~Rham's exterior derivative are defined on currents by duality and we also have $d = \partial + \bar{\partial},$ and set $d^c = \frac{\partial - \bar{\partial}}{2i\pi},$ to obtain $dd^c = \frac{\partial \bar{\partial}}{i\pi}.$ An important family of $(n-1,n-1)$-dimensional, or $(1,1)$-bidegree, positive currents are currents of the form $dd^c \psi$ where $\psi$ is a \textit{plurisubharmonic} function. Recall that a function $\psi:\Omega \to [-\infty, \infty)$, for $\Omega\subseteq \C^n$ an open set, is called plurisubharmonic if
\begin{itemize}
  \item [(a)] $\psi$ is not identically $-\infty$ in any component of $\Omega;$
  \item [(b)] $\psi$ is upper semicontinuous;
  \item [(c)] for every complex line $L\subseteq \C^n,$ $\psi_{{\upharpoonright}_{\Omega \cap L}}$ is subharmonic on $\Omega \cap L.$
\end{itemize}
\vskip 2mm
\noindent
It is well-known that when $u:\R^p \to \R$ is convex and increasing in each variable and $\psi_1, \dots ,\psi_p$ are plurisubharmonic, then
$$
u(\psi_1, \dots , \psi_p)
$$
is also a plurisubharmonic function. Moreover, for any holomorphic function $f:\C^n \to \C,$ $\log|f|$ is plurisubharmonic, and one has the following equality of currents, for which we provide a tropical version in Subsection~\ref{sec:trop_Lelong-Poincare}.
\begin{theorem}[Lelong--Poincar\'e Equation]\label{thm:LP}
Let $f$ be a non-zero meromorphic function on $X,$ and let $\sum m_j Z_j$ be the divisor of $f.$ The function $\log|f|$ is locally integrable on $X,$ and
$$
dd^c \log|f| = \sum m_j [Z_j].
$$
\end{theorem}

Given an analytic subset  $E\subseteq X$, and a positive closed $(p,p)$-current $\sT$ in $\sD'_{p,p}(X \setminus E)$ it is important to know when $\sT$ can be extended by zero to a closed positive current $\overline{\sT} \in \sD'_{p,p}(X).$ The theorem of El~Mir--Skoda asserts that  this extension is possible if $\sT$ has a finite mass in a neighbourhood of every point of $E$; see {\cite[III.2.3]{DemaillyBook1}}. 
% The following lemma will be useful in the extension of our convergence results from tori to toric varieties. Informally, it asserts that, under El~Mir--Skoda hypothesis, if a sequence does not \emph{charge} an analytic set, then nor does its weak limit.

% \begin{lemma}\label{lem:Elmir-Skoda2}
% Let $X$ be a K\"ahler manifold and $E\subseteq X$ an analytic cycle. Assume that
% \begin{itemize}
%   \item [(a)] We have a weakly convergent sequence of the positive closed currents
%   $$\sT_i \longrightarrow \sT, \quad \text{in }\cD'_{p,p}(X \setminus E);$$
%   \item [(b)] All $\sT_i$'s and $\sT$ have finite mass in a neighbourhood of every point of $E.$
% \end{itemize}
% Then
%  $$
%  \overline{\sT}_i \longrightarrow \overline{\sT}, \quad \text{in } \cD'_{p,p}(X).
%  $$
% \end{lemma}
% \begin{proof}
% Since the statement is local, we can consider a small open set $\Omega$ instead of $X.$ In the proof of the El~Mir--Skoda theorem {\cite[III.2.3]{DemaillyBook1}}, one constructs a smooth family of plurisubharmonic functions $u_k$ on $\Omega,$ such that
% $$
% \overline{\sT}_i = \lim_{k\to +\infty } u_k \sT_i,\quad \overline{\sT} = \lim_{k\to +\infty } u_k \sT.
% $$
% We write,
% $$
% \overline{\sT}_i - \overline{\sT} = (\overline{\sT}_i- u_k \sT_i) + (u_k \sT_i - u_k \sT)+ (u_k \sT - \overline{\sT}).
% $$
% The first and last differences converge to zero by definition. For the middle pair, we note that $u_k$'s are smooth and conclude.

% \end{proof}

\vskip 2mm
The \emph{pushforward} and \emph{pullback} of a current  are defined as a dual operation to pullback and  \emph{pushforward}  of differential forms, respectively. The pushforward of a form is defined by integrating along fibers: consider a submersion $f:X\longrightarrow X'$ between the complex manifolds $X$ and $X'$ with respective complex dimensions $m$ and $n.$ Let $\varphi$ be a differential form of degree $k$ on $X$, with $L^1_{\text{loc}}$ coefficients such that the restriction $f{\upharpoonright}_{\supp{\varphi}}$ is proper. Then,
$$f_{*}(\varphi):= \int_{z\in f^{-1}(y)}\varphi(z) \in \mathscr{D}^{k-2(m-n)}(X').$$
In this situation, for $\sT\in \mathscr{D}_{k-2(n-m)}'(X)$ one defines $f^*\sT\in \mathscr{D}_{k}'(X),$ given by
$$
\langle f^*(\sT), \varphi  \rangle := \langle \sT, f_*(\varphi) \rangle.
$$
Note that for an analytic cycle $Z$, we have $f^*[Z]= [f^{-1}(Z)]$.
\vskip 2mm
The set of seminorms 
$$|\cdot|_{\varphi}:  \sD'_{p,q}(X)\longrightarrow \R,\quad  \sT\longmapsto |\langle \sT, \varphi \rangle|,$$ 
for all $\varphi \in \sD^{p,q}(X)$, 
defines a topology on $\sD'_{p,q}(X).$ With respect to this \emph{weak topology}, the  operations 
$$
\sT \longmapsto d\, \sT,\quad \sT \longmapsto d^c\, \sT, \quad \sT\longmapsto f_* \sT, \quad \sT\longmapsto f^* \sT,
$$
are \emph{weakly continuous}. See \cite[Section I.2.C.3]{DemaillyBook1}.
\vskip 2mm
We finally recall that the wedge product or intersection of two positive closed currents is not always admissible, as we cannot always multiply measures. The leading theories are due to Bedford--Taylor \cite{Bedford--Taylor} and Demailly \cite{DemaillyRegularization} in codimension one, which we will use in Section \ref{sec:trop-toric}. 
In higher dimensions, the theory has been developed in \cites{Dinh-Sibony:superpot, Dinh-Sibony-Density}, \cite{Dinh-Nguyen-Vu} and \cite{Yger-et-al}. 
% \vskip 2mm
% By Bedford--Taylor theory, the wedge product of a bidegree $(1,1)$ positive closed current $dd^c u$ and a $(p,p)$ positive closed current $\sT,$
% $$
% dd^c~u \wedge \sT := dd^c (u \sT),
% $$
% is admissible, as soon as $u$ is locally integrable with respect to the trace measure of $\sT.$

\section{Tropical Algebra and Tropical Cycles}\label{sec:trop_algebra}
 Tropical geometry can be viewed as a geometry over  \emph{max-plus} or \emph{min-plus} algebra. In this article, we choose $(\mathbb{T}, \otimes, \oplus),$ where $\T = \R\cup \{-\infty \},$ and $\otimes$ and $\oplus$ are the usual sum and maximum, respectively. These operations can be obtained from the following multiplication and addition through the \emph{Maslov dequantisation}: for $x,y \in \R$, let $x \otimes_h y := x +y$ and $x \oplus_h y := h \ln (\exp (x/h)+ \exp(y/h)),$ as $h\to 0.$ See \cite{Litvinov} for a survey of the topic  touching on questions in thermodynamics, classical and quantum physics, and probability. See also \cites{Mikh-pants, Mikh-trop-enum} and surveys \cites{Viro-Hilbert, Itenberg-Mikhalkin}  on how Maslov dequantisation and degeneration of amoebas have played a fundamental role in the application of tropical geometry to enumerative problems in algebraic geometry. 
\vskip 2mm
In comparison to Maslov dequantisation, for $z, w \in \C^*,$ we can consider $\log|zw|= \log|z|+\log|w|,$ and apply $\frac{1}{m} \Phi^*_m$ to $\log|z + w|$ to obtain:
\begin{equation*}\label{eq:phi_max}
  \frac{1}{m} \Phi^*_m \big(\log|z+w|\big) = \frac{1}{m}\log|z^m + w^m| \xrightarrow[m\to \infty]{} \max \log \{ |z|, |w|\},
\end{equation*}
in the sense of \emph{distributions} or bidegree $(0,0)$ currents. We use this observation to derive a dynamical version of Kapranov's theorem in Subsection~\ref{sec:Kapranov}. In these notes, we follow \cite{Maclagan-Sturmfels} as our main reference of tropical geometry and refer the reader  to \cites{Gian-Gian, Maclagan--Rincon, Maclagan--Rincon-scheme, Lorscheid} for  recent scheme-theoretic development of tropical algebraic geometry.

\subsection{Tropical Cycles}
Recall that a linear subspace $H\subseteq \R^n$ is called \emph{rational} if it is spanned by a subset of $\Z^n.$ A \emph{rational polyhedron} in $\R^n$ is an intersection of finitely many rational half-spaces which are defined by
$$
\{x\in \R^n: \langle m , x \rangle \geq c, \text{ for some } m\in \ZZ^n,~ c \in \mathbb{R} \}.
$$
A \emph{rational polyhedral complex} is a complex with only rational polyhedra. The polyhedra in a polyhedral complex are also called \textit{cells.} A \textit{fan} is a polyhedral complex whose cells are all cones. If any cone of a fan $\Sigma$ also belongs to  another fan $\Sigma',$ then $\Sigma$ is a \emph{subfan} of $\Sigma'.$  The \emph{dimension} of a polyhedron is the dimension of the affine subspace of the minimal dimension containing it. All the fans and polyhedral complexes considered in this article are {rational}. 

\vskip 2mm
For a given polyhedron $\sigma$ let $\text{aff}(\sigma)$ be the affine span of $\sigma,$ and $H_{\sigma}$ be the translation of $\text{aff}(\sigma)$ to the origin. Assume that $\tau$ is a face of codimension one for  the  $p$-dimensional polyhedron $\sigma,$ and $u_{\sigma / \tau }$ is the unique outward generator of the one dimensional lattice $(\Z^n \cap H_{\sigma})/(\Z^n \cap H_{\tau}).$

\begin{definition}[Balancing Condition and Tropical Cycles]\label{BalancingCondition}
Let $\mathcal{C}$ be a rational polyhedral complex of pure dimension $p.$ Assume that all the cells of $\cC$ of dimension $p$ are weighted with positive integers. We say that $\cC$  satisfies the \emph{balancing condition} at $\tau$ if
\[
\sum_{\sigma\supset \tau} w_{\sigma}~ u_{\sigma/\tau}=0, \quad \text{in } \Z^n /(\Z^n \cap H_{\tau}),
\]
where the sum is over all $p$-dimensional cells $\sigma$ in $\mathcal{C}$ containing $\tau$ as a codimension  $1$ face, and $w_{\sigma}$ is weight of $\sigma$. A weighted complex is \emph{balanced} if it satisfies the balancing condition at each of its codimension one cells. A \emph{tropical cycle} or synonymously a  \emph{tropical variety} in $\R^n$ is a balanced weighted complex with finitely many cells.
\end{definition}
For a tropical variety $\cC$ we denote its support by $|\cC|,$ which is the underlying set of the $\cC$ as a polyhedral complex. In the tropical algebra, \emph{tropical polynomials} are obtained by addition and multiplication of the variables and constants in the tropical semiring and we can associate to each tropical polynomial a tropical hypersurface:
\begin{definition}\label{def:tropical_poly}
For a tropical polynomial
\begin{align*}
  \mathfrak{q}: \R^n &\longrightarrow \R, \\
  x &\longmapsto \max \{\langle x, \alpha \rangle + c_{\alpha} \},
\end{align*}
where $\alpha \in A\subseteq \Z^n_{\geq 0},$ $|A|< \infty,$ and $c_{\alpha}\in \R,$ the associated tropical variety to $\mathfrak{q}$, denoted by $V_{\trop}(\mathfrak{q}),$ is defined by:
\begin{itemize}
  \item [(a)] as a set $$
  |V_{\trop}(\mathfrak{q})| = \{x\in \R^n: \mathfrak{q}(x)\text{~is not differentiable at~} x\};$$
  \item [(b)] the weights on an $(n-1)$-dimensional cell $\sigma$ of $|V_{\trop}(\mathfrak{q})|$ is given by the lattice length of $\alpha_1 - \alpha_2,$ where $ |\sigma| = \{x\in \R^n: \langle \alpha_1  , x \rangle +c_{\alpha_1} = \langle \alpha_2 , x \rangle +c_{\alpha_2} =\mathfrak{q}(x) \}.$
\end{itemize}
\end{definition}
Given a tropical polynomial $\mathfrak{q},$ one can verify that $V_{\trop}(\mathfrak{q})$ is indeed balanced, {\cite[Proposition 3.3.2]{Maclagan-Sturmfels}}.

\section{Complex Tropical Currents}\label{sec:trop_currents} 
Introduction of different notions of \emph{tropical currents} as bridges between the theory of currents and tropical geometry, to the author's understanding, is indebted to \cite{DemaillyHodge, Rash, Pass-Rull}: Consider $(S^1)^n$-invariant plurisubharmonic functions of the form $\Log_+^*f = f\circ \Log_+$ for a convex continuous function $f:\R^n \longrightarrow \R,$ and
$$
\Log_+: (\CC^*)^n \longrightarrow \R^n, \quad (z_1, \dots, z_n) \longmapsto (\log|z_1|, \dots , \log|z_n|).
$$
By Rashkovskii's formula \cite{Rash}, one has
$$
\int_{\Log_+^{-1}(V)} (dd^c f \circ \Log_+)^n = n!~ \textrm{MA}(f)(V),
$$
where $\text{MA}(f)$ is the \textit{real Monge-Amp\`ere measure} of $f$ and $V \subseteq \R^n$ is a Borel set. Recall that when $f$ is smooth, we have
$$
\text{MA}(f)(V) = \det(\text{Hess}(f)) \mu,
$$
where $\mu$ is the Lebesgue measure in $\R^n$ and the definition can be extended to the non-smooth case by continuity. Passare and Rullg{\aa}rd in \cite{Pass-Rull}, used a mixed form of Rashkovskii's formula to reprove the Bernstein theorem. Lagerberg in \cite{Lagerberg}, introduced the theory of \emph{superforms} and \emph{supercurrents} where he could also decompose the real Monge--Amp\`ere measures into bidegree $(1,1)$ components, $dd^\sharp f$, to obtain
$$
\int_{\Log_+^{-1}(V)} (dd^c f \circ \Log_+)^n = \int_{V \times \R^n} (dd^\sharp f)^n.
$$
Moreover, Lagerberg introduced \emph{tropical supercurrents} in codimension one, which were generalised by Gubler in \cite{Gubler} to higher codimensions. Chambert-Loir and Ducros in \cite{Chambert-Ducros}  followed by Gubler and K\"unnemann in \cite{Gubler--Kunnemann} extended supercurrents to a fully-fledged theory on Berkovich spaces, and used them in application to Arakelov theory. Further, in \cite{BHJK2}, and in \cite{Mihatsch}  intersection theory and Monge--Amp\`ere operators were investigated; see also \cites{BMPS, CGSZ, Ducros-Hrushovski-Loeser} for several far-reaching extensions of Rashkovskii's formula. In \cites{Gil-Gubler-Jell-Kunnemann}, it was proved that the push-forward of the logarithm map can be defined in order to find a correspondence between a  cone of certain $(S^1)^n$-invariant closed positive currents on a complex toric variety and closed positive supercurrents on the tropicalisation of that toric variety. Moreover, one can lift, by the pull-back of the logarithm map, the tropical currents in \cite{Gubler} to \emph{complex tropical currents}.
\vskip 2mm
Complex tropical currents were introduced in \cite{Babaee} with a geometric representation which made it  convenient for generalizing Demailly's important example of an \textit{extremal} current  that is not an integration current along any analytic set. Recall that this current is given by $dd^c \max \{x_0,x_1, x_2\} \, \circ \, \Log_+ \in \mathscr{D}'_{2,2}(\P^2),$ and equals the (complex) tropical current associated to the tropical line in $\R^2$; see \cite{DemaillyHodge}. The extremality results in \cite{Babaee} were subsequently improved and extended to toric varieties by Huh and the author in \cite{BH}, and tropical currents were used to find a non-trivial example of a positive closed current on a smooth projective toric variety which refutes a strong version of the Hodge conjecture for positive currents. Thereafter, Adiprasito and the author in \cite{Adi-Baba} proposed a family of  tropical currents which are counter-examples to the aforementioned conjecture in any dimension and codimension greater than one. Moreover, tropical currents were also used in application to higher convexity problems and Nisse--Sottile conjecture \cite{Nisse-Sottile}, as well as finding a family of peculiar currents which cannot be regularised to obtain mollified currents with smooth boundaries; see {\cite[Theorem D]{Adi-Baba}}.
\vskip 2mm
In all the above-mentioned works though, an intrinsic notion of tropicalisation of integration currents along algebraic varieties was absent, which is the topic of this article. 
\subsection{Tropical Currents}
In this subsection, we recall the definition of tropical currents and their basic properties. We refer the reader to \cite{Babaee, BH} for more details. Let $N$ be a finitely generated free abelian group, we define
\begin{eqnarray*}
T_N&:=&\text{the complex algebraic torus $\mathbb{C}^* \otimes_\mathbb{Z} N$,}\\
S_N&:=&\text{the compact real torus $S^1 \otimes_\mathbb{Z} N$,}\\
N_\mathbb{R}&:=&\text{the real vector space $\mathbb{R} \otimes_\mathbb{Z} N$.}
\end{eqnarray*}
Let $\mathbb{C}^*$ be the group of nonzero complex numbers. As before, the logarithm map is the homomorphism
\[
\text{Log}: (\mathbb{C}^*)^n \longrightarrow \mathbb{R}^n, \qquad (z_1, \dots, z_n) \longmapsto (-\log |z_1|,\dots , -\log|z_n|),
\]
and the \emph{argument map} is
\[
\text{Arg}: (\mathbb{C}^*)^n \longrightarrow (S^1)^n, \qquad (z_1, \dots, z_n) \longmapsto (z_1/|z_1|,\dots , z_n/|z_n|).
\]
For a rational linear subspace $H \subseteq \RR^n$ we have the following exact sequences:
\[
\xymatrix{
0 \ar[r]& {H \cap \mathbb{Z}^n} \ar[r]& \mathbb{Z}^n \ar[r] & \Z^n(H) \ar[r] & 0,
}
\]
where ${\Z^n(H)}:= \mathbb{Z}^n/(H \cap \mathbb{Z}^n).$ Moreover,

\[
\xymatrix{
0 \ar[r]& S_{H \cap \mathbb{Z}^n} \ar[r]& (S^1)^n = S^1\otimes_{\mathbb{Z}} \mathbb{Z}^n \ar[r] &S_{\Z^n(H)} \ar[r] & 0.
}
\]

Define
\[
\pi_H:\xymatrix{ \text{Log}^{-1}(H) \ar[r]^{\quad\text{Arg}}& (S^1)^n \ar[r]& S_{\Z^n(H)}}.
\]
One has
\[
\text{ker}(\pi_H)=T_{H \cap \mathbb{Z}^n} \subseteq (\mathbb{C}^*)^n.
\]
As a result,  when $H$ is of dimension $p,$ the set $\text{Log}^{-1}(H)$ is naturally foliated by copies of $T_{H \cap \mathbb{Z}^n}  \simeq (\C^*)^p.$
\begin{definition}
Let $H$ be a rational subspace of dimension $p,$ and $\mu$ be the Haar measure of mass $1$ on $S_{\Z^n(H)}$. We define a $(p,p)$-dimensional closed current $\mathscr{T}_H$ on $(\mathbb{C}^*)^n$ by
\[
\mathscr{T}_H:=\int_{x \in S_{\Z^n(H)}} \big[\pi_H^{-1}(x)\big] \ d\mu(x).
\]
\end{definition}
\begin{example}\label{Ex:hypersurface}
It is instructive to explicitly understand the equations of the fibers of $\sT_{H},$ where $H\subseteq \R^n,$ is a rational hyperplane given by $H=\{x\in \R^n: \langle \beta, x\rangle =0\}$ for some $ \beta\in \Z^n.$
The support of $\sT_{H}$ is given by $\Log^{-1}(H).$ Let $\beta= w_{H} \alpha,$ where $\alpha \in \Z^n$ is a \emph{primitive} vector, \textit{i.e.}, its components have greatest common divisor equal to $1$ and $w_{H} \in \Z_{>0}$. In this case, each fiber $\pi_H^{-1}(x)$ is given by the $(S^1)^n$-translations of the toric set
$$
\pi_{H}^{-1}(1) =\{z \in (\C^*)^n: z^{-\alpha}-1 = 0 \} = \{z \in (\C^*)^n: z^{\alpha_+} - z^{\alpha_-}=0 \},
$$
 where $\alpha = \alpha^+ - \alpha^-,$ and $\alpha^{\pm}\in \Z^n_{\geq 0}.$ 

%More explicitly, for $\bar{x} \in S_{\Z^n(H)},$ regarded as by 
% $$\big[\pi_{H}^{-1}(x)\big]=dd^c \log|(x\cdot z)^{-\alpha}-1|,$$
% where $x\cdot z$ is the component-wise multiplication in $(\C^*)^{n}.$
We also note that the tangent space of the fiber $\pi_{H}^{-1}(1)$ at $w= (1, \dots, 1) \in (\C^*)^n$ is given by  $$T_{w}\pi_{H}^{-1}(1)= \ker \nabla (z^{-\alpha}-1)(w)= H\otimes_{\Z} \C.$$ 
It follows that if $H, H'\subseteq \R^n$ are two hyperplanes intersecting transversely in $\R^n,$  then all the fibers of $\sT_H$ and $\sT_{H'}$ intersect transversely in $(\C^*)^n.$
\end{example}
When $A$ is an affine subspace of $\mathbb{R}^n$ parallel to the linear subspace $H= A - {a}$ for $a \in A$, we define $\sT_{A}$ by translation of $\sT_H.$ Namely, we define the submersion $\pi_{A}$ as the composition
\[
\pi_{A}:\xymatrix{\text{Log}^{-1}(A) \ar[r]^{e^{a} } & \text{Log}^{-1}(H) \ar[r]^{\pi_{H} }&S_{\Z^n(H)}.}
\]
For $\sigma$ a $p$-dimensional (rational) polyhedron in $\R^n,$ we denote
\begin{align*}
  \text{aff}(\sigma) &:= \text{the affine span of $\sigma$,} \\
  \sigma^{\circ} &:= \text{the interior of $\sigma$ in $\text{aff}(\sigma)$,} \\
  H_{\sigma}&:= \text{the linear subspace parallel to $\text{aff}(\sigma)$}, \\
  N(\sigma) &:= \Z^n \slash (H_{\sigma} \cap \Z^n),
\end{align*}
and for homogeneity, $N(H):= \Z^n(H) = \Z^n \slash (H \cap \Z^n).$
\begin{definition}
Let $\mathcal{C},$ be a weighted polyhedral complex of dimension $p$. The tropical current $\mathscr{T}_{\mathcal{C}}$ associated to ${\mathcal{C}}$ is given by
$$
\mathscr{T}_{\mathcal{C}}= \sum_{\sigma} w_{\sigma} ~ \mathbbm{1}_{\Log^{-1}(\sigma^{\circ})}\mathscr{T}_{\text{aff}(\sigma)} ,
$$
where the sum runs over all $p$-dimensional cells $\sigma$ of $\cC.$
\end{definition}

\begin{theorem}[\cite{Babaee}]\label{thm:closed-balanced}
A weighted complex $\mathcal{C}$ is balanced, if and only if, $\mathscr{T}_{\mathcal{C}}$ is closed.
\end{theorem}

The reminder of this section is devoted to proving the statements that will be useful in later sections. The idea of the following technical lemma is extensively used in the proof of Theorem~\ref{thm:closed-balanced} and extremality results in \cites{Babaee,BH}. For a measure $\eta$ on the $k$-dimensional compact torus $(S^1)^k$ let us denote by $\widehat{\eta}(\ell_1, \dots, \ell_k)$ its $(\ell_1, \dots, \ell_k)$-th Fourier measure coefficients.
\begin{proposition}\label{prop:haar_measures}
For a $p$-dimensional affine plane $A$ parallel to a rational linear space, assume that the current $\sT_A(\eta)$ is given by
$$
\sT_{A}(\eta)=\int_{x \in S_{N(H)}} \big[ \pi_{A}^{-1}(x)\big] \ d\eta(x),
$$
for a positive measure $\eta.$ Then, all the Fourier measure coefficients $\widehat{\eta}(\nu),$ for $\nu \in \Z^n,$ can be determined by the action of the current $\sT_{A}(\eta)$ on the $(p,p)$-differential forms of type
\begin{equation*}\label{eq:fourier-forms}
\omega= \exp(-i\langle \nu , \theta \rangle ) \rho(r) d\theta_{I} \wedge dr_{I},
\end{equation*}
where $I \subseteq [n],$ with $|I|=p,$ $\theta= (\theta_1, \dots, \theta_n)$ and $r= (r_1, \dots, r_n)$ are polar coordinates, and $\rho:\RR^n \to \RR$ is a smooth function with compact support.
\end{proposition}
\begin{proof}[Sketch of Proof.]
Without loss of generality, we can assume that $A$ is a linear subspace of $\R^n$. Let us choose $\{w_1, w_2, \dots, w_p\}$ a $\ZZ-$basis for $\ZZ^n \cap A $ and extend it to a $\ZZ$-basis of of $\ZZ^n,$ $B:=\{w_1, \dots, w_p, u_1, \dots, u_{n-p}\}.$ One can see that, compare to \cite[Equation (3.3.20)]{Babaee}, with the right choice of $\rho$,
$$
\langle \sT_{A}(\eta), \omega \rangle = \delta_{\{\nu \in A^{\perp}\}} ~\widehat{\eta}(\langle u_1, \nu \rangle, \dots , \langle u_{n-p}, \nu \rangle) ~\text{det}_I(w_1, \dots , w_p),
$$
where $\text{det}_I$ is the minor of rows corresponding to $I\subseteq [n],$ and
$$\delta_{\{\nu \in A^{\perp}\}} =
\begin{cases}
1 & \text{if } \nu \in A^{\perp};\\
0 & \text{otherwise,}
\end{cases}$$
where $A^{\perp}$ is the orthogonal complement of $A$ in $\R^n$ with respect to usual dot product. As $B$ is a $\ZZ$-basis, the tuple $(\langle u_1, \nu \rangle, \dots , \langle u_{n-p}, \nu \rangle)$ can assume any value in $\Z^{n-p}$ and therefore $\eta$ can be fully understood.

\end{proof}

 We call a current on $(\C^*)^n$, $(S^1)^n-$invariant if it is invariant under the induced action of multiplication by any $x\in (S^1)^n.$ The following proposition is comparable to {\cite[Example 7.1.7]{Gil-Gubler-Jell-Kunnemann}}. See also \cite[Theorem 3.6]{CGSZ} for an interesting related result in codimension one.

\begin{proposition}\label{prop:charac}
Assume that $\sT\in \sD'_{p,p}((\C^*)^n)$ is a closed positive $(S^1)^n$-invariant current whose support is given by $\Log^{-1}(|\cC|)$, for a polyhedral complex $\cC\subseteq \R^n$ of pure dimension $p$. Then $\sT$ is a tropical current.

\end{proposition}
\begin{proof}
By Demailly's second theorem of support, \cite[III.2.13]{DemaillyBook1}, any positive current with support $\Log^{-1}(|\cC|)$ can be presented as
$$
\sT = \sum_{\sigma \in \cC} \int_{x \in S_{N(H)}} \mathbbm{1}_{\Log^{-1}(\sigma^{\circ})}\big[\pi_{\text{aff}(\sigma)}^{-1}(x)\big] \ d\eta_{\sigma}(x),
$$
for a unique complex measure $\eta_{\sigma}$ for each $\sigma$. We need to see that the measures $\eta_{\sigma}$ are indeed Haar measures, but this is simple, since by definition of the fibers $\pi_A$ each fiber of $\pi_A$ is a translation of the kernel by the action of $(S^1)^n$, and if $\sT$ is invariant under $(S^1)^n$ then each $\eta_{\sigma}$ has to be invariant under the translation in the quotient $S(N(H))= S_{\Z^n/(H\cap \Z^n)}.$
\end{proof}

Before ending this subsection let us recall the notion of refinement.

\begin{definition}\label{def:refinement}
A $p$-dimensional weighted complex $\cC'$ is called a refinement of $\cC,$ if $|\cC|= |\cC'|,$ and each cell $\sigma'\in \cC'$ of dimension $p$ (and non-zero weight) is contained in some $p$-dimensional cell $\sigma \in \cC$ with
$$
w_{\sigma'}(\cC') = w_{\sigma}(\cC).
$$
\end{definition}
It is easy to check that if $\cC$ and $\cC'$ have a common refinement then their tropical currents $\sT_{\cC}$ and $\sT_{\cC'}$ coincide; see {\cite[Section 2.6]{BH}}.

\section{The Toric Setting}\label{sec:trop-toric}
\subsection{The Multiplication Map on Toric Varieties}
In this subsection, we intend to analyse the dynamics of $\Phi_{m}$ on the points of toric varieties. The reader may consult \cite{Cox-Little-Schenck} for a thorough and the notes \cite{Brasselet} for a quick introduction to the theory of toric varieties. 
%In this article, we make the novel change of variables $m = \ell.$
\vskip 2mm
As before, let $N$ be a free abelian group of rank $n$, $M= \Hom_{\Z}(N,\Z),$
$$
\quad N_{\R}= \R \otimes_{\Z} N,\quad T_N = \C^* \otimes_{\Z} N, \quad S_N = S^1 \otimes_{\Z} N.
$$
For a rational cone $\sigma,$ we let $N_{\sigma}$ be the sublattice of $N$ spanned by the points in $\sigma \cap N,$ and $N(\sigma)= N \slash N_{\sigma}.$ Moreover,
\begin{align*}
   M_{\R} &:= \R \otimes_{\Z} M, \\
  \sigma^{\vee} &:= \{u \in M_{\R}: \langle u , v \rangle \geq 0, \text{ for all } v \in \sigma \}, \\
  \sigma^{\perp} &:= \{u \in M_{\R}: \langle u , v \rangle = 0, \text{ for all } v \in \sigma \}.
\end{align*}
\vskip 2mm
Let us fix a rational fan $\Sigma \subseteq N_{\R},$ and recall that $\Sigma$ defines a toric variety $X_{\Sigma}$ which is obtained by gluing the affine varieties $U_{\sigma}:= \Spec~\C[\sigma^{\vee} \cap M]$ according to the fan structure in~$\Sigma.$ 
\vskip 2mm
There is a \emph{Cone-Orbit Correspondence} between the cones in $\Sigma$ and orbits of the continuous action of $T_N$ on $X_{\Sigma}.$ Namely, for any integer $0\leq p \leq n,$
\begin{align*}
\{\text{$p$-dimensional cones $\sigma$ in $\Sigma$}\} &\longleftrightarrow
\{\text{$(n-p)$-dimensional $T_{N}$-orbits in $X_{\Sigma}$}\} \\
\sigma & \longleftrightarrow \cO(\sigma)\simeq T_{N(\sigma)}.
\end{align*}
Therefore, the $T_N$-orbits give rise to the partition 
$$
X_{\Sigma} = \bigcup_{\sigma \in \Sigma} \cO(\sigma).
$$

Each point of any affine piece $U_{\sigma}$ corresponds to a semigroup homomorphism
$$
\sigma^{\vee} \cap M \longrightarrow \C,
$$
where $\C$ is considered as a semigroup under multiplication. The \emph{distinguished point} $z_{\sigma}\in \cO(\sigma)\subseteq U_{\sigma},$ is then  defined to be the unique point that corresponds to the semigroup homomorphism
$$
u \in \sigma^{\vee} \cap M ~\longmapsto
\begin{cases}
1 & \text{if $u \in \sigma^{\perp}$}, \\
0 & \text{if $u \notin \sigma^{\perp}$}.
\end{cases}
$$
The distinguished point of the open torus $(\C^*)^n,$ for instance, is just $(1,\dots, 1).$
We also have a Lie group isomorphism $T_{N(\sigma)} \longrightarrow \cO(\sigma)$, that can be understood by $t~\longmapsto~t \cdot z_{\sigma},$ and we can devise it to define the \emph{distinguished compact torus} of $\cO(\sigma),$ given by
$$S(\sigma):= S_{N(\sigma)}\cdot z_{\sigma}\, .$$
\vskip 2mm
Now, for any positive integer $m$, we consider the multiplication map
$$
{\phi}_{m}: N \longrightarrow N, \quad v\longmapsto m v,
$$
which induces the \emph{toric morphism} (see \cite[Definition 1.3.13]{Cox-Little-Schenck}),
$$
\phi_m \otimes_{\Z} 1: T_{N} \longrightarrow T_{N}, \quad t \longmapsto t^m.
$$
To justify the latter, note that in terms of semigroup homomorphisms the induced map $\phi_m \otimes_{\Z} 1,$ is indeed
$$
\gamma \longmapsto \gamma \circ \phi_m = \gamma^m,
$$
for any semigroup homomorphism $\gamma: M \longrightarrow \C.$ As a result, when $T_N= (\C^*)^n$, we obtain our familiar group endomorphism
$$
\Phi_{m}= \phi_m \otimes_{\Z} 1: (\C^*)^n \longrightarrow (\C^*)^n, \quad (z_1, \dots, z_n) \longmapsto (z_1^{m}, \dots, z_n^{m}).
$$
\vskip 2mm 
We can also extend our multiplication map to $(\phi_m)_{\R}:N_{\R}\longrightarrow N_{\R},$ and note that it maps any cone $\sigma\in \Sigma$ to itself, and as a result it induces a {toric endomorphism}, which we also denote by $\Phi_{m}: X_{\Sigma}\longrightarrow X_{\Sigma};$ see \cite[Theorem 3.3.4]{Cox-Little-Schenck}. Since $\sigma^{\circ}$, the relative interior of $\sigma$, is invariant under $(\phi_m)_{\R},$ every toric orbit $\cO(\sigma)$ also remains invariant under $\Phi_m.$ In addition, $(\phi_{m})_{\R}^{-1}(\sigma^{\circ}) = \sigma^{\circ}$ implies that $\Phi_m^{-1}(\cO_{\sigma}) = \cO_{\sigma}.$ 
\vskip 2mm
Let us understand $\Phi_m{\upharpoonright}_{\cO(\sigma)}$ explicitly  by observing that
\begin{align*}
\Phi_m {\upharpoonright}_{\cO(\sigma)}:\cO(\sigma) &\longrightarrow \cO(\sigma), \\
t \cdot z_{\sigma} &\longmapsto t^m \cdot z_{\sigma}, \quad t\in T_{N(\sigma)}.
\end{align*}
Note that $T_N \simeq \Hom_{\Z}(M, \C^*),$ and $\Phi_m$ is the continuous extension of $t\in T_N \longmapsto t^m \in T_N,$ to $X_{\Sigma}\longrightarrow X_{\Sigma}.$ Moreover, the isomorphism $\cO(\sigma)\simeq \Hom_{\Z}(\sigma^{\perp} \cap M, \C^*),$ implies that $\cO(\sigma)$ as a subset of $U_{\sigma} = \Spec~\C[\sigma^{\vee}\cap M],$ is given by the set of semigroup homomorphisms satisfying
\begin{align*}
\gamma: \sigma^{\vee}\cap M &\longrightarrow \C, \\
u &\longmapsto
\begin{cases}
\gamma(u)\in \C^* & \text{if $u \in \sigma^{\perp}$}, \\
0 & \text{if $u \notin \sigma^{\perp}$}.
\end{cases}
\end{align*}
As a result, the map $\Phi_m: T_N \longrightarrow T_N,$ $\gamma \longmapsto \gamma^m,$ extends continuously on $\cO(\sigma)~\longrightarrow~\cO(\sigma),$ by $\gamma \longmapsto \gamma^m.$ Finally, we may write $\gamma = t\cdot z_{\sigma},$ for some $t\in T_{N(\sigma)},$ to obtain
$$\gamma = t\cdot z_{\sigma} \longmapsto \gamma^m = (t\cdot z_{\sigma})^m = t^m \cdot z_{\sigma}.$$
%This obervation also implies that if we consider $T_{N(\sigma)} \simeq (\C^*)^{n-p} \times 0_{\C^p} \subseteq \overline{(\C^*)^n } \simeq \oT_{N}$

Now we have gathered enough tools to show that we have a separate {equidistribution theorem} within each toric orbit:

\begin{proposition}\label{prop:equi_orbit}
Let $\Phi_{m}:X_{\Sigma} \longrightarrow X_{\Sigma}$ be the toric endomorphism induced by the multiplication map. For any $z \in \cO(\sigma),$ we have the weak convergence
$$
m^{\dim(\sigma)-n} \Phi_{m}^{*}(\delta_z) \longrightarrow \mu(S(\sigma)),\quad \text{as $m\to \infty,$}
$$
where $\mu(S(\sigma))$ is the normalised Haar measure on the distinguished compact torus $S(\sigma)= S_{N(\sigma)}\cdot z_{\sigma} \subseteq \cO(\sigma).$
\end{proposition}
\begin{proof}
We have observed that
\begin{align*}
\Phi_m {\upharpoonright}_{\cO(\sigma)}:\cO(\sigma) &\longrightarrow \cO(\sigma), \\
t \cdot z_{\sigma} &\longmapsto t^m \cdot z_{\sigma}, \quad t\in T_{N(\sigma)}.
\end{align*}
Let
\begin{align*}
  \xi: T_{N} &\longrightarrow \cO(\sigma),\quad t \longmapsto t\cdot z_{\sigma}, \\
  \varphi_m: T_{N(\sigma)} &\longrightarrow T_{N(\sigma)},\quad t\longmapsto t^m.
\end{align*}
The fact that the pullbacks $\xi^*$ and $(\xi^{-1})^*$ are continuous with respect to the weak topology of currents, see Section \ref{sec:prem_currents}, implies that to prove the proposition it suffices to observe the weak convergence for the conjugate map instead:
$$
m^{\dim(\sigma)-n}\varphi_m^* (\delta_{z'}) \longrightarrow \mu(S_{N(\sigma)}),
$$
for any $z' \in T_{N(\sigma)},$ with $z' = \xi^{-1}(z).$ Let $q= \dim(\cO(\sigma))= n-\dim(\sigma).$ Note that the isomorphism $N(\sigma) \simeq \Z^q,$ induces an isomorphisms of Lie groups $S_{N(\sigma)}\simeq (S^1)^q$, and a biholomorphism of complex manifolds
$$
N(\sigma)\otimes_{\Z} \C^* \xrightarrow[~]{~~\thicksim~~} \Z^{q}\otimes_{\Z} \C^* = (\C^*)^q.
$$
Consequently, by weak continuity of pullback morphisms, to prove the asserted convergence, it suffices to show the analogous statement on $(\C^*)^q.$ On $(\C^*)^q$, however, the induced map, which we also denote by $\Phi_m,$ is given by
$$
(\C^*)^q \longrightarrow (\C^*)^q, \quad (z_1, \dots, z_q) \longmapsto (z_1^{m}, \dots, z_q^{m}).
$$
In this case, for any compactly supported smooth function $f:(\C^*)^n \longrightarrow \C$, we have
$$
m^{-q}\langle \Phi_{m}^*(z), f \rangle \longrightarrow \int_{(S^1)^q} f d\mu,
$$
as the left-hand side is simply tending to the Riemann sum for the integral on the right-hand side since $\mu$ is the Haar measure on $(S^1)^q$.

\end{proof}

Let us end this subsection with the following useful lemma, which can be also proved directly for any toric variety without the smoothness assumption.

\begin{lemma}\label{lem:compactify-commute}
Let $Z\subseteq (\CC^*)^n$ be a subvariety, and $\bar{Z}$ be the closure of $Z$ in a smooth toric variety $X_{\Sigma}.$ Then
$$
\Phi_m^{-1}(\bar{Z}) = \overline{{\Phi^{-1}_{m}}(Z)}\, .
$$
\end{lemma}
% Note that the assertion in the above lemma is, in fact, a \emph{flatness property}, see \cite[Th\'eor\`eme 2.3.10]{GrothendieckIV}, and it is not implied by the properness of the map alone. For instance, consider the blow-up $\varphi: X\longrightarrow \mathbb{C}^n$ at the origin $0\in \mathbb{C}^n$; see \cite[Page 28]{Hartshorne}. The map $\varphi$ is indeed proper, however, for any line $L$ passing through the origin, we certainly have the \emph{strict transform} strictly included in the preimage of $L$, \textit{i.e.},
% $$\overline{\varphi^{-1}(L\setminus \{0 \})} \subsetneq \varphi^{-1}(\overline{L\setminus \{0 \}}).$$

\begin{proof}
When $X_{\Sigma}$ is smooth, according to \cite[Exercise 18.17]{Eisenbud-commutative-book}, the finiteness property of $\Phi_m$ implies that $\Phi_m$ is flat. The assertion then follows from \cite[Th\'eor\`eme 2.3.10]{GrothendieckIV}.

% However, we can deduce the statement and in general as follows. We first note that the continuity of $\Phi_m$ readily implies the inclusion $\Phi_m^{-1}(\bar{Z}) \supseteq \overline{\Phi^{-1}_{m}(Z)}.$ To see the converse, assume that $\Phi_m^{-1}(\bar{Z}) \setminus \Phi_m^{-1}(Z)\neq \varnothing,$ and take $y \in \Phi_m^{-1}(\bar{Z}) \setminus \Phi_m^{-1}(Z).$ There exists a cone $\sigma \in \Sigma$ with $y \in \cO(\sigma)$ as well as a sequence $\{x_n\}_{n\in \mathbb{N}} \subseteq Z$ converging to $\Phi_m(y) \in \bar{Z}\cap \cO(\sigma).$ Properness of the map $\Phi_m$ implies that the preimage of the compact set $\{x_n \}_{n\in \mathbb{N}}\cup \{ \Phi_m(y)\}$ given by $\{\Phi_m^{-1}(x_n) \}_{n\in \mathbb{N}}\, \cup \, \big\{ \Phi_m^{-1}\big(\Phi_m(y)\big)\big\}$ is also compact. We have
% \begin{align*}
% \big\{ \Phi_m^{-1}(z_{\{0\}})\cdot y_n \big\} &= \big\{\Phi_m^{-1}(x_n) \big\} \quad \text{for any $y_n \in \{\Phi_m^{-1}(x_n) \}$, and} \\
% \big\{ \Phi_m^{-1}(z_{\sigma})\cdot y \big\} &=\big\{ \Phi_m^{-1}\big(\Phi_m(y)\big)\big\},
% \end{align*}
% where $z_{\{0\}}$ and $z_{\sigma}$ are distinguished points of $T_N $ and $\cO(\sigma),$ respectively. Now, choose any sequence $\{y'_n\}_{n\in \mathbb{N}}\subseteq \big\{ \Phi_m^{-1}(z_{\{0\}})\cdot y_n \big\}_{n\in \mathbb{N}}$ converging to a point in $y' \in \big\{ \Phi_m^{-1}(z_{\sigma})\cdot y \big\} ,$ then by symmetry of $m$-th roots, 
% $$
% y \in \big\{ \Phi_m^{-1}(z_{\sigma})\cdot y' \big\} \subseteq \overline{ \big\{ \Phi_m^{-1}(z_{\{0\}})\cdot y_n \big\}}.
% $$

\end{proof}

\subsection{Tropical Compactifications}
Let $X$ be a smooth, projective toric variety and fix a torus equivariant projective embedding 
$$
\phi: X \longrightarrow \P^N. 
$$
Assume that $\omega_0$ is the smooth positive $(1,1)$-form on $X,$ corresponding to $\phi.$ We have that by Wirtinger's theorem {\cite[Page 31]{Griffiths-Harris}}, the closure of any closed algebraic variety $Z\subset (\C^*)^n$ in $X$ has a finite mass with respect to $\omega_0$. 
% For any closed algebraic subvariety $Z\subseteq (\C^*)^n$, we can consider its closure $\bar{Z}\subset \P^n.$ Then $\bar{Z}$ has a finite degree with respect to the standard K\"ahler form on $\P^n,$ and as a result of Wirtinger's theorem {\cite[Page 31]{Griffiths-Harris}} a finite normalised (global) mass, see also {\cite[Page 171]{Griffiths-Harris}}. 
We can therefore employ the El~Mir--Skoda theorem, \cite[Section III.2.A]{DemaillyBook1}, and extend $[Z]$ by zero to the closed positive current $\overline{[Z]}$ on $X$. Similarly, any tropical current $\sT_{\cC}$ can be extended by zero to any smooth projective toric variety, since any fiber $\pi_{\textrm{aff}(\sigma)}^{-1}(x),$ for $\sigma \in \cC,$ has a bounded normalised mass; see \cite[Proposition 4.4]{BH} for more details where we also treat the tropical currents with non-positive weights. If, moreover, we demand that the support of  $\overline{[Z]}$ or the fibers of $\overline{\sT}_{\cC}$ intersect the torus-invariant divisors of $X_{\Sigma}$ properly we need certain compatibility with $\Sigma$, which we present in Theorems~\ref{thm:tevelev} and \ref{thm:weight_intersect} below, after recalling some important theorems from toric geometry.
\vskip 2mm
 Recall that a $p$-dimensional cone $\sigma = \text{cone}(\rho_1, \dots , \rho_p) \subseteq \R^n \simeq N_{\R}$ is called \emph{unimodular}, if $\{\rho_1, \dots, \rho_p \}$ can be completed to a $\Z$-basis of $\Z^n.$  
\begin{theorem}\label{thm:toric_basic}
Let $\Sigma\subseteq \R^n,$ be a fan. Then,
\begin{enumerate}
  \item [(a)] $X_{\Sigma}$ is smooth, if and only if, $\Sigma$ is \emph{unimodular}, \textit{i.e.}, all cones in $\Sigma$ are unimodular.
  \item [(b)] $X_{\Sigma}$ is \emph{complete}, if and only if, $\Sigma$ is \emph{complete}, \textit{i.e.}, $|\Sigma| = \R^n.$
  \item [(c)] $X_{\Sigma}$ is projective, if and only if, $\Sigma$ is \emph{projective}, \textit{i.e.}, $\Sigma$ is \emph{dual} to a polytope.
\end{enumerate}
\end{theorem}
Recall that, since our base field is the set of complex numbers,  a variety is \emph{complete}, if and only if, it is compact in the classical topology. Moreover, we do not require our varieties to be projective in this article, and have added Theorems \ref{thm:toric_basic}.c and \ref{thm:toric_adv}.c for completeness; see \cite[Page 67]{Cox-Little-Schenck} for the definition of a dual polytope.  Further,
\begin{theorem}\label{thm:toric_adv}
\begin{enumerate}

  \item [(a)] The Toric Resolution of Singularities: any fan $\Sigma$ can be refined to obtain a unimodular fan;  \cite[Theorem 11.1.9]{Cox-Little-Schenck}.
  \item [(b)] For any fan $\Sigma$ there exists complete fan $\Sigma'$ containing $\Sigma$ as a subfan; \cite[Theorem III.2.8]{Ewald}.  
  \item [(c)] The Toric Chow Lemma: every complete fan has a refinement that is projective; \cite[Theorem 6.1.18]{Cox-Little-Schenck}.
\end{enumerate}
\end{theorem}
Since our main objects in this article are differential forms and currents, we always need our ambient toric varieties to be smooth or equivalently their fan to be unimodular. When $X_{\Sigma}$ is smooth, $X_{\Sigma} \setminus (\C^*)^n$ is a simple normal crossing divisor, and the orbit closures $D_{\sigma}:= \overline{\cO(\sigma)},$ for $\sigma \in \Sigma,$ are intersections of its irreducible components.
\vskip 2mm
As we mentioned in the introduction that $\trop(Z)$ as a set is given by the Hausdorff limit
$$\trop(Z):= \lim_{t\to \infty}\frac{1}{\log|t|}\Log(Z).$$
Further, we know from the Structure Theorem in tropical geometry, \cite[Theorem 3.3.5]{Maclagan-Sturmfels}, that $\trop(Z)$ can be regarded as the support of a rational polyhedral complex. In the following paragraphs, we recall how to endow a tropical structure on the set $\trop(Z),$ using the following theorem due to Tevelev and Sturmfels. Tevelev's theorem below provides a condition for the closure of $Z$ to be complete in a given toric variety, even when the toric variety is not complete. We also note that one significance of \cite{Tevelev} is to consider the case where the compactification is \emph{flat}, see \cite[Section 6.4]{Maclagan-Sturmfels} for a definition, which we do not require in this article.  
\begin{theorem}\label{thm:tevelev}
Let $Z \subseteq (\C^*)^n$ be an irreducible algebraic subvariety of dimension $p$, and $\Sigma\subseteq N_{\R}$, a unimodular (rational) fan.
\begin{itemize}
  \item [(a)]  The closure $\bar{Z}$ of $Z$ in $X_{\Sigma}$ is complete, if and only if, $\trop(Z) \subseteq |\Sigma|;$ see \cite{Tevelev}. 
  \item [(b)] We have $|\Sigma|= \trop(Z)$, if and only if, for every $\sigma \in \Sigma$ the intersection $ \cO_{\sigma}\cap \bar{Z}$ is non-empty and of pure dimension $p- \dim(\sigma);$ see \cite{Sturmfels-Tevelev}.
\end{itemize}
\end{theorem}
For a subvariety $Z\subseteq (\C^*)^n$, we can always find a unimodular fan $\Sigma$ such that $\trop(Z)=|\Sigma|,$ this is a consequence of the toric resolution of singularities, Thereom~\ref{thm:toric_adv}.a. Suppose now that $\Sigma'$ is a fan that contains $\Sigma$ as a subfan. Then, the closure of $Z$ in $X_{\Sigma}= \bigcup_{\sigma \in \Sigma} \cO(\sigma)
$ can be identified with the closure of $Z$ in $X_{\Sigma'} = \bigcup_{\sigma \in \Sigma'} \cO(\sigma).$ Applying the preceding theorem we obtain that for a $p$-dimensional cone $\sigma\in \Sigma',$ the intersection $\cO(\sigma) \cap \bar{Z}$ is non-empty, if and only if, $\sigma \in \Sigma,$ and in this case, the intersection is of dimension $p- \dim(\sigma);$ see also \cite[Section 6.4]{Maclagan-Sturmfels} and \cite{Gubler-Guide}. Therefore, thanks to \cite[Corollary 4.10]{DemaillyBook1}, for any positive integer $k\leq p,$ if $\tau= \text{cone}(\rho_1, \dots , \rho_k)\in \Sigma,$ where $\rho_i \in \Sigma$ are rays, the wedge product
%then for any $\sigma \in \Sigma,$ $\cO(\sigma)\subset X_{\Sigma}$ can be identified with $\cO(\sigma)\subset X_{\Sigma'}$ and
$$[D_{\tau}]\wedge [\bar{Z}] = [D_{\rho_1}]\wedge \dots \wedge [D_{\rho_k}] \wedge [\bar{Z}]$$
is admissible, and yields a $(p-k, p-k)$-closed positive current which can be considered in both $X_{\Sigma}$ and $X_{\Sigma'}$. When $\dim(\sigma) = p,$ the intersection is $0$-dimensional, and we set
$$
w_{\sigma}:= \int_{X_{\Sigma}} [D_\sigma] \wedge [\bar{Z}].
$$

It is observed in tropical geometry that $\trop(Z)$ with the induced fan structure from $\Sigma,$ and the induced weights $w_{\sigma}$ is a balanced fan; see \cite[Theorem 6.7.7]{Maclagan-Sturmfels}. We also review this induced balancing condition as a consequence of our main convergence theorem by noting that the weak limit of a sequence of closed currents is indeed closed.
\begin{definition}\label{def:trop_comp}
Suppose $Z\subseteq (\C^*)^n$ is an irreducible subvariety, and $\Sigma$ is a fan with $|\Sigma| = \trop(Z).$ The closure $\bar{Z}$ of $Z$ in the toric variety $X_{\Sigma},$ or equivalently in any $X_{\Sigma'},$ such that $\Sigma'$ contains $\Sigma$ as a subfan, is called a \emph{tropical compactification} of $Z.$
\end{definition}

Let us also investigate the situation for the tropical currents. Recall from Theorem \ref{thm:closed-balanced} when a polyhedral complex $\cC$ satisfies the balancing condition, then $\sT_{\cC}$ is closed. 
\begin{theorem}\label{thm:weight_intersect}
Let $\cC, \Sigma \subseteq \R^n$ be two fans, and assume that $\cC$ is a $p$-dimensional tropical variety. 
\begin{itemize}
 \item [(a)]  Let $\tau\in \cC,$ $\sigma \in \Sigma,$ and $\overline{\pi^{-1}_{{\text{aff}(\tau)}}(x)}$ be the closure in $U_{\sigma}.$ Then, the intersection
 $$
 D_{\sigma} \cap \overline{\pi^{-1}_{{\text{aff}(\tau)}}(x)}
 $$
 is non-empty, if and only if, $\tau$ contains $\sigma$ as a face and in this case the intersection is transverse;  {\cite[Lemma 4.10.1]{BH}}.
 \item [(b)] If $\sigma \in \cC \cap \Sigma$ is cone of dimension $p$ with weight $w_{\sigma}$ in $\cC$, then
 $$
 w_{\sigma} = \int_{X_{\Sigma}} [D_{\sigma}] \wedge \overline{\sT}_{\cC}.
 $$
where $\overline{\sT}_{\cC}$ is the extension by zero of $\sT_{\cC}$ in $X_{\Sigma}$; \cite[Theorem 4.7]{BH}.
\end{itemize}
\end{theorem}
The preceding theorem asserts that the wedge product 
$\overline{[\pi_{\text{aff}(\sigma)}^{-1}(x)]} \wedge  \, D_{\sigma},$ for any $x\in S_{N(\sigma)}$ is a well-defined $(0,0)$-dimensional current. Transversality implies that the  intersection multiplicity is one, and Part (b) implies that averaging with respect to the Haar measure and multiplying with $w_{\sigma}$ will indeed yield the total mass $w_{\sigma}$. 
%We need the following proposition is derived from {\cite[Proposition 4.12]{BH}}.

% be a unimodular subdivision of $|\trop(Z)|,$ which always exists by the toric %resolution of singularities \cite[Theorem 11.1.9]{Cox-Little-Schenck}.

%recession cone $\text{rec}(\sigma)=\sigma.$ Assertion (2), follows easily from assertion~(1) and %observing that in $\overline{\sT}_{\cC}$ the fibers are added with respect to the normalised Haar %measure and have weights $w_{\sigma}.$

% The following proposition, implies Theorem~\ref{thm:intro_toric} if we assume that Theorem~\ref{thm:main_intro}, holds.

% \begin{proposition}\label{prop:extention}
% Assume that $\sZ_{i}$ is subsequence of $m^{p-n}\Phi_m^*[Z],$ converging to a tropical current $\sS$ with support $\Log^{-1}(\trop(Z)).$ Then, $\overline{\sZ}_i \longrightarrow \overline{\sS},$ where the overline denotes extension by zero of the currents to $X_{\Sigma}.$

% where $\Phi_m: X_{\Sigma}\longrightarrow X_{\Sigma}$ is the extension of $\Phi_m: (\C^*)^n \longrightarrow (\C^*)^n$ and $\overline{\mathscr{T}}_{\trop(Z)}$ is the \emph{extension by zero} of $\mathscr{T}_{\trop(Z)}$ to $X_{\Sigma}.$

% \end{proposition}

% Let $\cC$ be a $p$-dimensional tropical variety in $\RR^n$ with finitely many cells. We say $\cC$ is \emph{compatible} with $\Sigma$ if its underlying fan is a subfan of $\Sigma.$

\begin{remark}\label{rem:proper_for_all}
Let $Z\subseteq (\C^*)^n$ be an irreducible subvariety and consider $\Sigma$ a unimodular fan with $\trop(Z)= |\Sigma|.$ Since, for any positive integer $m$
$$\Log(\Phi^{-1}_m(Z)) = \frac{1}{m}\Log(Z),$$
 the sets $\trop(\Phi_m^{-1}(Z))$ coincide. As a result, by Theorem~\ref{thm:tevelev} all $\overline{\Phi_m^{-1}(Z)}$'s are simultaneously complete in $X_{\Sigma},$ and intersect the toric invariant divisors of $X_{\Sigma}$ properly. In consequence, the currents $\big[\,\overline{\Phi_m^{-1}(Z)}\,\big]$ carry no mass on $X_{\Sigma}\setminus (\C^*)^n.$ And, 
$$
 \overline{\big[\Phi_m^{-1}(Z)\big]}=\big[\,\overline{\Phi_m^{-1}(Z)}\,\big],\quad \text{for $m \in \Z_{\geq 0}$},
$$
where $\Phi_0^{-1}(Z):= Z.$ Finally, in view of Lemma~\ref{lem:compactify-commute}, we conveniently obtain
%\begin{equation}\label{eq:convenience}
$$
  \overline{\big[\Phi_m^{-1}(Z)\big]} = \big[\,\overline{\Phi_m^{-1}(Z)}\,\big] = [\Phi_m^{-1}(\bar{Z})].
$$
%\end{equation}
%for any tropical compactification of $Z.$
\end{remark}
% Therefore, for any algebraic subvariety $Z\subset \C^*$by Theorem~\ref{thm:tevelev}, and the fact that $\Log^(\Phi_m^{-1}(Z)) = $
% \vskip 2mm
%Let us observe that the tropical compactification commutes with $\Phi_m^{-1}.$

%\subsection{Tevelev's Compactification}
%In this section we explain how to extend the map %$\Phi_m$ and tropical currents from $(\C^*)^n$ to %any toric variety.

%As before, we consider the lattice $N \simeq %\Z^n$ a finitely generated abelian group, with %dual $M= \Hom_{\Z}(N,\Z).$ A fan $\Sigma$ in %$N_{\R}= N\otimes \R,$ defines a toric variety

%For an algebraic variety $Z\subseteq (\C^*)^n$, %the set $\Log(Z)$ is called the \textit{amoeba %of} $Z.$ by Bergman's Theorem \cite{Bergman}, in %the Hausdorff metric in the compact sets of %$\R^n,$ one has the limit
%$$
%\frac{1}{\log|t|}\Log(Z) \to \cC,\quad %\text{as\,\,} t \to \infty,
%$$
%for a rational polyhedra complex $\cC.$ See also %\cite{Jonsson} for more details. It is shown by %Speyer \cite{speyer-thesis} that $\cC$ can be %naturally equipped with integer weights to make %$\cC$ balanced, and in this case, we call %$\trop(Z):=\cC$ the \textit{tropicalisation of} %$Z.$
\subsection{The Extended Tropicalisation}
For any rational fan $\Sigma \subseteq N_{\R},$  the logarithm map $\Log: T_N \longrightarrow N_{\R},$ can be continuously extended to
$$
\Log: X_{\Sigma} \longrightarrow N_{\Sigma},
$$
where $N_{\Sigma}$ denotes the quotient $X_{\Sigma}\slash S_{N(\{0\})};$ see \cite[Section 3]{Jonsson}. The \emph{tropical toric variety} $N_{\Sigma}$ can be regarded as the tropicalisation of toric variety $X_{\Sigma}$ and it was studied in  \cites{AMRT, Kajiwara, Payne} and \cite[Section 6.2]{Maclagan-Sturmfels}. The set $N_{\Sigma}$ has a natural topology and it is straightforward to see that for a polyhedral complex $\cC\subseteq N_\R \subseteq N_{\Sigma},$ 
$$
\Log^{-1}(\overline{\cC}) = \overline{\Log^{-1}(\cC)}.
$$

By \cite[Theorem A$'$]{Jonsson} for an algebraic subvariety $Y \subseteq X_{\Sigma},$ the logarithmic tropicalisation also extends to $X_{\Sigma}:$
$$
\frac{1}{\log|t|}\Log(Y) \longrightarrow \trop(Y),\quad \text{as $|t| \to \infty$},
$$
where $\trop(Y)\subseteq N_{\Sigma}$ is the tropicalisation of $Y$ in $N_{\Sigma}.$ Now assume that $Z \subseteq T_N$ is an algebraic subvariety, and $\bar{Z}$ is the closure of $Z$ in $X_{\Sigma},$ Theorem 6.2.18 in \cite{Maclagan-Sturmfels} implies that  $\trop(\bar{Z})$ coincides with the closure of $\trop(Z)\subseteq N_{\R}$ in $N_{\Sigma}.$ In consequence, 
\begin{equation}\tag{I}\label{eq:extended_log_lim}
\Log(\supp{\Phi_m^*[\bar{Z}]})= \frac{1}{m} \Log(\bar{Z}) \longrightarrow \trop(\bar{Z})= \overline{\trop(Z)},
\end{equation}
where the convergence is in the Hausdorff metric of compact sets of $N_{\Sigma}$.

\section{A Proof for  Theorem~\ref{thm:main_intro}}\label{sec:main_proof}
We are now ready to prove the main theorem of this article. The reader may find viewing the figure in Subsection~\ref{sec:toric_tori} instructive while following the steps of the proof.

\begin{theorem}\label{thm:main_conv}
Let $Z\subseteq (\CC^*)^n$ be an algebraic variety of dimension $p,$ then we have the weak convergence
$$
m^{p-n}\Phi_m^*[Z] \longrightarrow \mathscr{T}_{\trop(Z)},
$$
in $\sD'_{p,p}((\C^*)^n).$ Moreover, with the induced weights from any unimodular tropical compactification of $Z$, $\trop(Z)$ is balanced.
\end{theorem}

\begin{proof}
We proceed in the following steps.
\vskip 2mm

{Step 1.} The extention of the map $\Phi_m$ to an endomorphism of $\P^n$ has degree $m,$ with respect to $O(1).$ It is well-known that $m^{p-n}$ is the correct normalisation factor for $\Phi_m^*$ to preserve the normalised total mass of the closure of $[Z]$ in $\P^n$. Therefore, all the elements of the sequence $\{m^{p-n}\Phi^*_m[\bar{Z}] \}$ have an equal normalised total mass, and as a result of the Banach--Alaoglu theorem, the sequence is locally compact with respect to the weak topology of currents. In other words, it has a convergence subsequence in $\P^n$, and by restriction on $(\C^*)^n.$
\vskip 2mm

 {Step 2.} In Subsection ~\ref{subs:conv_supp} below, we prove that the support of any cluster value $\sS$ of the sequence is given by
  $\Log^{-1}(\trop(Z)).$ Let $\Sigma$ be a unimodular fan with $|\Sigma|= \trop(Z).$ Since $\sS$ is a positive closed current of bidimension $(p,p)$, by Demailly's second theorem of support, \cite[III.2.13]{DemaillyBook1}, it can be represented as
  $$
  \sS = \sum_{\sigma \in \Sigma} \int_{x \in S_{N(\sigma)}} \big[ \mathbbm{1}_{\Log^{-1}(\sigma^{\circ})}\pi_{\mathrm{aff}(\sigma)}^{-1}(x)\big] \ d\eta_{\sigma}(x),
  $$
  for some positive measures $d\eta_{\sigma}.$
  \vskip 2mm
 {Step 3.} In Subsection~\ref{subs:weaklimit_trop} below, we show that for any cluster value $\sS$ obtained in the above steps, the corresponding measures $d\eta_{\sigma}$ must be Haar measures. This, together with Proposition~\ref{prop:charac}, implies that any cluster value $\sS$ is indeed a tropical current with support $\Log^{-1}(\trop(Z)).$
 \vskip 2mm
  Step 4. Let $\sZ_{j_i}$ be a subsequence of  $m^{p-n}\Phi_m^*[Z]$ weakly convergent to the tropical current $\sS,$ we show that in the toric variety $X_{\Sigma},$ we have
  $ \overline{\sZ}_{j_i}\longrightarrow \overline{\sS}.$
  As in Step 2, the support of $\sS$ is given by $\Log^{-1}(\trop(Z)).$ Therefore 
  $$\supp{\overline{\sS}} \supseteq\overline{\Log^{-1}(\trop(Z))} = \Log^{-1}(\overline{\trop(Z)}) = \Log^{-1}({\trop(\bar{Z})}),$$
  where in the first equality we have used a simple property of $\Log: X_{\Sigma} \longrightarrow N_{\Sigma},$ and the second equality is provided by 
  \cite[Theorem 6.2.18]{Maclagan-Sturmfels}. Assume now that $\widetilde{\sS}$ is a cluster value of $\overline{\sZ}_{j_i}.$ 
  By Equation (\ref{eq:extended_log_lim}) and Remark \ref{rem:supp_haus}, $\Log(\supp{\widetilde{\sS}}) \subseteq {\trop(\bar{Z})},$ therefore 
  $$ \supp{\widetilde{\sS}} \subseteq \Log^{-1}({\trop(\bar{Z})}) \subseteq \supp{\overline{\sS}}.$$ 
  Since $\overline{\sS}{\upharpoonright}_{(\C^*)^n} = \widetilde{\sS}{\upharpoonright}_{(\C^*)^n},$ the support of the positive closed $(p,p)$-dimensional current $\widetilde{\sS}- \overline{\sS}$ is included in $\bigcup_{i} D_i\cap \supp{\overline{\sS}},$ where $D_i$'s are the torus invariant divisors  of $X_{\Sigma}.$ In view of Theorem \ref{thm:weight_intersect}, $\bigcup_{i} D_i\cap \supp{\overline{\sS}}$ has a Cauchy--Riemann dimension less than $p,$ and by Demailly's first theorem of support, \cite[Theorem III.2.10]{DemaillyBook1}, we have $\widetilde{\sS}= \overline{\sS}.$ 
  \vskip 2mm 
  Step 5. Assume that $\mathscr{Z}_{\ell_i}$ and $\mathscr{Z}_{k_j}$ are two subsequences of $m^{p-n}\Phi_m^*[Z]$ weakly convergent to the tropical currents $\sS_1$ and $\sS_2$, respectively. We intend to show that $\sS_1 = \sS_2.$ Applying Theorem~\ref{thm:toric_adv}.b, and the toric resolution of singularities Theorem~\ref{thm:toric_adv}.a, we can find a fan $\Sigma'$ containing $\Sigma$ as a subfan such that $X_{\Sigma'}$ is a smooth projective toric variety. In view of Theorem~\ref{thm:weight_intersect}.b, it only remains to show that for any $p$-dimensional cone $\sigma \in \Sigma',$
  $$
  w_{\sigma}(\overline{\sS}_1) := \int_{X_{\Sigma'}} [D_{\sigma}] \wedge \overline{\sS}_1 \quad \overset{\mathrm{~~}}{=\joinrel=} \quad w_{\sigma}(\overline{\sS}_2):= \int_{X_{\Sigma'}} [D_{\sigma}] \wedge \overline{\sS}_2,
  $$
  where the transversality of  intersections are guaranteed by Step~2 and Theorem~\ref{thm:weight_intersect}.  For $\sigma \in \Sigma',$ let
  $$
  w_{\sigma}(\bar{Z}) = \int_{X_{\Sigma'}} [D_{\sigma}] \wedge [\bar{Z}].
  $$
  We show that $w_{\sigma}(\bar{Z})= w_{\sigma}(\overline{\sS}_1) = w_{\sigma}(\overline{\sS}_2)$, for any $\sigma \in \Sigma'.$ By Step 4, in $X_{\Sigma'}$
 
% , and the core idea based on the powerful regularisation theorem of Demailly %\cite{DemaillyRegularization}. Let $\Sigma$ be unimodular fan with $|\Sigma|= \Trop(Z),$ and %$X_{\Sigma}$ be the associated toric variety. In view of Theorem~\ref{thm:tevelev} and %Theorem~\ref{thm:weight_intersect}, the equality of the sets
% $$
%
$$
  \overline{\sZ}_{\ell_i}\longrightarrow \overline{\sS}_1, \quad \overline{\sZ}_{k_j}\longrightarrow \overline{\sS}_2.
  $$
  By Remark~\ref{rem:proper_for_all}, $$\overline{\sZ}_{\ell_i} = {\ell_i}^{n-p}\Phi^*_{\ell_i}[\bar{Z}], \quad \overline{\sZ}_{k_j} = {k_j}^{n-p}\Phi^*_{k_j}[\bar{Z}].$$
% Let
% $$
% w_{\sigma}(\bar{Z}) = \int_{X_{\Sigma}} [D_{\sigma}] \wedge [\bar{Z}].
% $$
  Since $\dim(\sigma)= p,$ we have $\dim \big(\cO(\sigma) \big)=n-p.$ As a result, the restriction $m^{p-n}\Phi^*_{m}{{\upharpoonright}_{\cO(\sigma)}}$ preserves the mass of zero dimensional integration currents in $\cO(\sigma).$ Hence, for any positive integer $m,$
  $$
  w_{\sigma}(\bar{Z}) = \int_{X_{\Sigma'}} m^{p-n}\Phi^*_{m}{{{\upharpoonright}_{\cO(\sigma)}}}([D_{\sigma}] \wedge [\bar{Z}]) = \int_{X_{\Sigma'}} [D_{\sigma}] \wedge \big(m^{p-n}\Phi^*_{m}[\bar{Z}]\big),
  $$
  where we have used the invariance of $D_{\sigma}$ under $\Phi_m^{-1}$ in the latter equality. Now we can use {\cite[Proposition 4.12]{BH}}, which is heavily based on Demailly's regularisation theorem, to reduce the calculations of the above masses to the intersection of cohomology classes. Namely, whenever $\sT$ is a closed positive current such that $ [D_{\sigma}] \wedge \sT$ is admissible,
  $$ \int_{X_{\Sigma'}} [D_{\sigma}] \wedge \sT = \big\{[D_\sigma] \wedge \sT \big\} = \{\omega_1\}\wedge \dots \wedge \{\omega_{p} \} \wedge \{\sT\}.
  $$
  In the above formula,
  \begin{itemize}
  \item [] $\{~\}$ denotes the Dolbeault cohomology class in $X_{\Sigma'},$
  \vskip 1mm
  \item [] $D_{\sigma} = D_1 \cap \dots \cap D_p$, is the intersection of toric invariant divisors, and
    \vskip 1mm

  \item [] $\omega_1, \dots , \omega_p$ are the first Chern forms corresponding to $D_1, \dots, D_p,$ respectively.
  \end{itemize}
 We note that we have chosen $X_{\Sigma'}$ to be a smooth compact complex manifold in order to  employ \cite[Proposition 4.12]{BH}. Now, using the weak continuity of the cohomology class assignment, and the above formulas,
\begin{align*}
  w_{\sigma}(\bar{Z})= \big\{[D_{\sigma}] \wedge \overline{\sZ}_{{\ell_i}} \big\} &\longrightarrow \big\{[D_{\sigma}] \wedge \overline{\sS}_1 \big\}= w_{\sigma}(\overline{\sS}_1), \\
  w_{\sigma}(\bar{Z})= \big\{[D_{\sigma}] \wedge \overline{\sZ}_{{k_j}} \big\} &\longrightarrow \big\{[D_{\sigma}] \wedge \overline{\sS}_2 \big\}= w_{\sigma}(\overline{\sS}_2).
\end{align*}
This concludes the proof of the main convergence theorem.

\vskip 2mm
 Step 6. Let us now observe that with the induced fan structure from $\Sigma$ or $\Sigma'$, and the induced weights, $\trop(Z)$ is balanced. The weak limit of a sequence of closed currents is a closed current, and Theorem~ \ref{thm:closed-balanced} asserts that a tropical current associated to any weighted polyhedral complex is closed, if and only if, the underlying weighted polyhedral complex is balanced.
\end{proof}

\subsection{Convergence of the Supports}\label{subs:conv_supp}
In this subsection we show that for any $p$-dimensional subvariety $Z \subseteq (\C^*)^n,$ the support of any cluster value of the sequence $m^{p-n}\Phi_m^*[Z]$ is given by $\Log^{-1}(\Trop(Z)).$ The inclusion in $ \Log^{-1}(\Trop(Z))$ is easy but for the converse, we need certain volume estimates. Notation-wise, it is slightly lighter to write the proofs of the following statements for the whole sequence rather than a subsequence, acknowledging that the proofs are identical.

%We now show the convergence of supports. This is given in the following proposition.

\begin{proposition}\label{Prop:Support}
Assume that $\sS$ is the weak limit of the sequence $\sZ_m: =m^{p-n}\Phi_m^*[Z].$ Then,
$$
\supp{\sS} = \Log^{-1}(\Trop(Z)).
$$
\end{proposition}
We give a proof of the preceding proposition after the proof of Lemma \ref{lem:local-mass-ball}.

\begin{corollary}\label{cor:support-Haus}
Assume that $\sS$ is the weak limit of the sequence $\sZ_m = m^{p-n}\Phi_m^*[Z],$ then
$$
\supp{\sZ_m} \longrightarrow \supp{\sT_{\trop(Z)}},
$$
in the Hausdorff metric of compact sets of $(\C^*)^n.$
\end{corollary}
\begin{proof}
We note that by Bergman's theorem and the definition of $\Phi_m$
  $$\supp{\sZ_m} \longrightarrow \Log^{-1}(\trop(Z)),$$
  in the Hausdorff metric and by Proposition~\ref{Prop:Support}, $\Log^{-1}(\trop(Z)) = \supp{\sT_{\trop(Z)}}.$
\end{proof}
%Let us finally prove a technical lemma.
\begin{lemma}\label{lem:local-mass-annulus} For any positive integer $m$ sufficiently large, small $\epsilon > 0,$ and $b_m \in \Phi_m^{-1}(Z)$ we have
$$
\int_{\Phi_{m}^{-1}\bigl(\Phi_m(B_{\epsilon}(b_m))\bigr)}m^{p-n}\Phi_m^*[Z] \wedge \beta ^ p \geq C \epsilon^{2p},
$$
where $\beta = dd^c \lVert z \rVert^2,$ and $C$ is a positive constant and independent of $m$.
\end{lemma}

\begin{proof} Let $b_m = (b_{m,1}, \dots, b_{m,n}), \, a_m= \Phi_m(b_m)=(a_{m,1}, \dots, a_{m,n})\in Z,$ and note that $|b_{m,j}| = |a_{m,j}|^{1/m}$. For each $j=1, \dots, n$, we set $\rmin_{j} := (|b_{m,j}|-\frac{\epsilon}{\sqrt{2}})$ and $\rmax_{j} := (|b_{m,j}|+\frac{\epsilon}{\sqrt{2}})$. For positive real numbers $\rmin<\rmax$, we will denote the annulus
\[C_{\rminm, \rmax}:= \{z\in \CC \,|\, \rmin<|z| <\rmax \}.\]
We also set
\[C_{\rminm, \rmax}^* := C_{\rminm, \rmax} \setminus \R_+\]
which is a contractible open subset of $C_{\rminm, \rmax}$. When $m$ is large, $\Phi_m(B_\epsilon(b_m))$ contains the multi-annulus given by
$$
B_m:=\prod_{j=1}^n C_{\rmin^m_{j}, \rmax^m_{j}}.
%\big((|b_i|-\frac{\epsilon}{\sqrt{2}})^m, (|b_i|+\frac{\epsilon}{\sqrt{2}})^m \big).
$$

Let $B_m^* = \prod_{j=1}^n C^*_{\rmin^m_{j}, \rmax^m_{j}}$. Since $\Phi_m$ is a covering and $B_m^*$ is contractible, we can partition $\Phi_{m}^{-1}\bigl(B_m^*)$ into a disjoint union of open sets $K_{m,1}, \dots , K_{m,m^n}$ such that for each $i \in [m^n]: = \{1, \dots, m^n\},$
$$
{\Phi_m}{{\upharpoonright}_{K_{m,i}}}:K_{m,i} \longrightarrow B_m^*,
$$
is a homeomorphism. Each $K_{m,i}$ can be parametrised in polar coordinates with $(0, {2\pi}/{m})^n \times \prod_{j\in [n]}(\rmin_j , \rmax_j).$
For a fixed $m$, and any $i\in [m^n]$,
$$
\int_{K_{m,i}} \Phi_m^*[Z] \wedge \beta^p = \int_{\Phi_m(K_{m,i})} [Z] \wedge (\Phi_m)_* \beta^p = \int_{B^*_m} [Z] \wedge (\Phi_m^{-1})^*(\beta^p).
$$
We set $Z_{i} = \Phi_m^{-1}(Z) \, \cap \, K_{m, i}$. Since by Lelong's theorem \cite[Theorem III.2.7]{DemaillyBook1}, the singular set of $Z$ does not charge any mass, we assume that $Z$ is smooth and consider the parametrisation
$$
\boldsymbol{\tau}:(S^1)^p   \times U \longrightarrow \Phi_m^{-1}(Z) \cap K_{m,1},
$$
for $U \subseteq (\RR^+)^p.$ Let $\{ 1, \zeta, ...,\zeta^{m-1}\}$ be all different $m$-th roots of unity, and for $\ell= (\ell_1, \dots, \ell_n)\in \ZZ^n \cap [0,m-1]^n,$ define $\lambda^{\ell}(\zeta):= (\zeta^{\ell_1}, \dots, \zeta^{\ell_n}).$ The component-wise multiplication $\boldsymbol{\tau}\cdot \lambda^\ell(\zeta)$ for different $\ell$ parametrises all the manifolds $\Phi_m^{-1}(Z) \cap K_{m,i}$ for $i=1, \dots, m^n.$ Therefore, the equality $$\big(\boldsymbol{\tau}\cdot \lambda^\ell(\zeta)\big)^* dd^c\lVert z\rVert^2 = \boldsymbol{\tau} ^*\,dd^c \lVert z\rVert^2,$$
implies that all $Z_i$'s have the same mass. We now claim that for each $i$
\begin{equation}\label{C-ell}\tag{II}
[Z_{i}] \wedge \beta^p \geq C_{i} \big(\frac{\epsilon^2}{m}\big)^p.
\end{equation}
for a constant $C_{i}.$ As we noted before, by Bergman's theorem and the definition of $\Phi_m$ we have $\supp{\sZ_m} \longrightarrow \Log^{-1}(\trop(Z)).$ As a result, by shrinking $\epsilon$, if necessary, there exists $\sigma \in \trop(Z),$ and a fiber $\pi^{-1}_{\sigma}(x)$ for some $ x\in (S^1)^{n-p},$ such that at a point $w\in \pi_{\sigma}^{-1}(x),$  the projection $\textrm{proj}: Z_i \longrightarrow T_w \pi_{\sigma}^{-1}(x)$ is surjective. It is crucial to note that, because of the Hausdorff convergence,  for larger $m$'s we do not need to shrink $\epsilon$ any further for such surjectivity to hold. Using this projection, we can find coordinates $w= (w', w'')$ and polydiscs $w'\in \Delta'\subseteq \C^p, w''\in \Delta''\subseteq \C^{n-p}$ of radii $r'=\frac{\epsilon}{\sqrt{2}}$ and $r''$ with $r'\leq Cr''$ for $C>0$ large, such that the projection $\pi: Z_{i}\,\cap\big(\Delta' \times \Delta'' \big)\longrightarrow\Delta'$ is also surjective. Therefore,
$$
\int_{Z_{i}} (dd^c \lVert w \rVert^2)^p \geq \int_{Z_{i}} \pi^* (dd^c \lVert w' \rVert^2)^p \geq \, \int_{\Delta'} (dd^c \lVert w' \rVert^2)^p.
$$
Since $K_{m,i}\supseteq Z_{i}$ can be parametrised by $(0, {2\pi}/{m})^n \times \prod_{j\in [n]}(\rmin_j , \rmax_j)$, we also can parametrise $\Delta',$ with $(0, {2\pi}/{m})^p \times (0, \frac{\epsilon}{\sqrt{2}})^p$, such that the Jacobian of the parametrisation does not depend on $m$. We have
$$
\int_{\Delta'} (dd^c \lVert w' \rVert^2)^p =\int_{(0, {2\pi}/{m})^p \times (0, \frac{\epsilon}{\sqrt{2}})^p} r_{[p]} dr_{[p]}d\theta_{[p]}=\big(\frac{\pi \epsilon^2}{2m} \big)^p.
$$
We obtain Equation~(\ref{C-ell}), by taking into account that the change of coordinates $z\mapsto w,$ contributes to the calculation of mass by multiplying a constant. Finally, to obtain the main assertion, we note that there are $m^n$ components of $K_{m,i},$ all with the same mass, and the normalising factor is $m^{p-n}$.

\end{proof}

\begin{lemma}\label{lem:local-mass-ball} For any positive integer $m$ sufficiently large, small $\epsilon > 0,$ and $b_m \in \Phi_m^{-1}(Z)$ we have
$$
\int_{B_{\epsilon}(b_m)}m^{p-n}\Phi_m^*[Z] \wedge \beta ^ p \geq C_{b_m}' \epsilon^{2p+n},
$$
where $\beta = dd^c \lVert z \rVert^2,$ and $C'$ is a positive constant depending on $b_m$.
\end{lemma}
\begin{proof}
Following the proof and the notation of Lemma \ref{lem:local-mass-annulus}, we find a lower bound for the mass of $\sZ_m= m^{p-n}\Phi_m^*[Z]$ in the ball $B_{\epsilon}(b_m).$ Let $b_m= (b_{m,1}, \dots , b_{m,n}).$ By previous lemma, the mass of $\sZ_m$ in the multi-annulus $A= \prod_{j=1}^n C_{\rmin_{j}, \rmax_{j}}$,  is greater than $C\epsilon^{2p}.$ We can assume that $|b_{m,j}| \gg \epsilon,$ for all $j=1,\dots, n.$ Note that for each $j,$  the annulus in the $j$-th coordinate of the multi-annulus has an angle varying between $0$ and $2\pi,$ which gives rise to the outer circumference of length $2\pi (|b_j|+ \frac{\epsilon}{\sqrt{2}}).$ Therefore, in each coordinate, the angle corresponding to $2\epsilon$ portion of the outer circumference can be approximated by  
$$\theta_j= \frac{2\pi\epsilon}{2(|b_{m,j}|+\frac{\epsilon}{\sqrt{2}})\pi}\geq  \frac{\epsilon}{2|b_{m,j}|}.$$ 
We intend to estimate the portion of mass of $\sZ_m$ around $b_m$ sliced within the angles $\theta_1,\dots, \theta_n.$ Since the mass of $\sZ_m$ in $A$ equals $C\epsilon^{2p},$ and components of $\Supp(\sZ_m)\cap A$  are divided symmetrically in $A,$ we obtain the following lower bound
$$
C\epsilon^{2p} \times \prod_{j=1}^{n} \frac{\theta_j}{2\pi} = C\epsilon^{2p}\prod_{j=1}^n  \big(\frac{\epsilon}{4\pi|b_{m,j}|} \big)^n = \frac{C}{(4 \pi)^n |b_{m,1}\dots b_{m,n}|} \epsilon^{2p+n}:=C_{b_m}\epsilon^{2p+n}.
$$

\end{proof}

\begin{remark}
As the above lemma suggests, the \emph{Lelong number} of any tropical current at any point of its support is, in fact, zero. See \cite[Section III.5]{DemaillyBook1} for a definition.
\end{remark}

\begin{proof}[Proof of Proposition \ref{Prop:Support}]
We first show the inclusion $\supp \sS \subseteq \Log^{-1}\big(\Trop(Z)\big).$ We note that if for any sufficiently large integer  $m$, $z$ has a neighbourhood that does not intersect $\Supp (\sZ_{m}),$ then $z \notin \Supp(\sS).$ However, by Bergman's theorem
$$
\Log \big(\supp{\Phi^*_m[Z]}\big) = \frac{1}{m} \Log \big(\supp{[Z]}\big) \longrightarrow \trop(Z),
$$
in the Hausdorff metric. Therefore, any point outside $\Log^{-1}\big(\trop(Z)\big)$ is outside of $\supp{\sZ_{m}}$ for any sufficiently large $m$. This implies that $\supp \sS \subseteq \Log^{-1}\big(\Trop(Z)\big).$
\vskip 2mm
To prove the converse inclusion $\Log^{-1}(\Trop(Z)) \subseteq \supp \sS,$
consider $b\in \Log^{-1}\big(\Trop(Z)\big).$ We show that for any $\epsilon>0$, there is a test form $\varphi \in \sD^{p,p}((\C^*)^n)$, such that its (compact) support contains the open ball of radius $2\epsilon$ centered at $b,$ $B_{2\epsilon}(b),$ and
$$
\langle \sS , \varphi \rangle  > 0.
$$
Let $\beta =dd^c \lVert z \rVert^2 = dd^c [|z_1|^2 + \dots + |z_n|^2].$ For any $\epsilon>0,$ and sufficiently large $m,$ there exists a point $$b_m \in \Phi_m^{-1}(Z) \subseteq \Log^{-1}\big(\Log(\Phi_m^{-1}(Z)) \big),$$
such that
$ \mathrm{dist}\big(b_m,b) < \epsilon.$
This implies the inclusion $B_{\epsilon}(b_m) \subseteq B_{2\epsilon}(b)$. By Lemma~\ref{lem:local-mass-ball}, we have that for any sufficiently large $m$, there exists a constant $C_{b_m}:=\frac{C}{(4 \pi)^n |b_{m,1}\dots b_{m,n}|}>0$ such that
$$
\int_{B_{2\epsilon}(b)} \sZ_{m} \wedge \beta^p > \int_{B_{\epsilon}(b_m)} \sZ_{m} \wedge \beta^p >C_{b_m}\epsilon^{2p+n}.
$$
As $b_m\to b,$ the sequence $\big\{C_{b_{m}}\big\} \cup \{C_b\}\subseteq \R$ is compact and has a minimum value, say $C',$ and we have
$$
\int_{B_{2\epsilon}(b)} \sZ_{m} \wedge \beta^p  \geq C' \epsilon^{2p+n},
$$
for all $m$. Now define the test-form $\varphi := \chi\beta^p,$ where $\chi$ is a non-negative smooth function with compact support, equal to $1$ in $B_{2\epsilon}(b)$ and vanishing outside $B_{3\epsilon}(b).$ As a result, the convergence of the complex sequence $\langle \sZ_{m}, \varphi \rangle \longrightarrow \langle  \sS, \varphi \rangle$, together with the above inequality implies that
$$
\langle \sS , \varphi \rangle \geq C' \epsilon^{2p+n}.
$$

\end{proof}

\subsection{Any Cluster Value is a Tropical Current}\label{subs:weaklimit_trop}
In Lemma~\ref{Prop:Support}, we used the volume form $(dd^c\lVert z \rVert^2)^p$ that adds the mass in different components $K_{m,1}, \dots, K_{m,m^n}$. In contrast, in the following lemma we use the ``Fourier differential forms'' which are sensitive to the change of phases in polar coordinates, and we show that Fourier measure coefficients finally vanish at every non-zero degree.
\begin{lemma}\label{lem:fourier}
Let $\sS$ be the weak limit of sequence $\sZ_{m}:=m^{p-n} \Phi^*_m[Z],$ then $\sS$ is a tropical current. \end{lemma}
\begin{proof}
By Proposition~\ref{Prop:Support}, the support of $\sS$ has the form $\Log^{-1}(\Trop(Z)).$ Let us set $\cC := \trop(Z).$ By Demailly's second theorem of support, \cite[III.2.13]{DemaillyBook1}, there are measures $\eta_{\sigma}$ such that
$$
\sS= \sum_{\sigma \in \cC} \int_{x \in S_{N(\sigma)}} \big[ \mathbbm{1}_{\Log^{-1}(\sigma^{\circ})}\pi_{\text{aff}(\sigma)}^{-1}(x)\big] \ d\eta_{\sigma}(x).
$$

By definition of tropical currents, we need to prove that all the non-zero degree Fourier coefficients of $\eta_{\sigma}$ for all $\sigma \in \Sigma$ vanish, and therefore are Haar measures. To see this we use Proposition~\ref{prop:haar_measures}. We prove this by showing that when $m$ is sufficiently large, 
$$
\langle \Phi_m [Z] , \omega_I \rangle =0,
$$
for any fixed $(p,p)$-differential form $\omega$ given in polar coordinates by
\begin{equation*}\label{eq:fourier-forms2}
\omega_I= \exp(-i\langle \nu , \theta \rangle ) \rho(r) d\theta_{I} \wedge dr_{I},
\end{equation*}
where $\nu\in \Z^n \setminus \{0\},$ $I \subseteq [n],$ with $|I|=p,$ $\theta= (\theta_1, \dots, \theta_n)$ and $r= (r_1, \dots, r_n)$ are polar coordinates, and $\rho:\RR^n \to \RR$ is a smooth and appropriately chosen, and $\omega$ has the support given by $\Phi_m^{-1}(B_{\epsilon}(a_m))$ for $a_m \in Z.$ For convenience, we change the coordinates to have $\omega_I = x^{\nu} \rho(r)~ dr_{I} \wedge dx_{I},$ with $(x, r) \in (S^1)^n \times (\RR^+)^n \simeq (\CC^*)^n.$
\vskip 2mm
Let us recycle the notation in the proof of Lemma~\ref{lem:local-mass-annulus}: $b_m = (b_{m,1}, \dots, b_{m,n}),\, a_m= \Phi_m(b_m)=(a_{m,1}, \dots, a_{m,n})\in Z,$ and we partition $\Phi_{m}^{-1}\bigl(B_m^*)$ into a disjoint union of open sets $K_{m,1}, \dots , K_{m,m^n}.$
Since by Lelong's theorem \cite[Theorem III.2.7]{DemaillyBook1}, the singular set of $Z$ does not charge any mass, we assume that $Z$ is smooth and consider the parametrisation
$$
\boldsymbol{\tau}: (S^1)^p \times U \longrightarrow \Phi_m^{-1}(Z) \cap K_{m,1},
$$
for $U\subseteq (\RR^+)^p.$ We let $\{ 1, \zeta, ...,\zeta^{m-1}\}$ be all different $m$-th roots of unity, and for $\ell= (\ell_1, \dots, \ell_n)\in \ZZ^n \cap [0,m-1]^n,$ define $\lambda^{\ell}(\zeta):= (\zeta^{\ell_1}, \dots, \zeta^{\ell_n}).$ The component-wise multiplication $\boldsymbol{\tau}\cdot \lambda^\ell(\zeta)$ for different $\ell$ parametrises all the manifolds $\Phi_m^{-1}(Z) \cap K_{m,i}$ for $i=1, \dots, m^n.$
We write
$$
\langle \Phi^*_m [Z], \omega_I \rangle = \sum_{i=1}^{m^n} \int_{\Phi_m^{-1}(Z) \,\cap \, K_{m,i}} \omega_I = \sum_{\ell \in \ZZ^n \cap [0,m-1]^n}\int_{(S^1)^p \times U} \big(\boldsymbol{\tau}\cdot\lambda^\ell(\zeta)\big)^* \omega_I.
$$
For each $\ell=(\ell_1, \dots , \ell_n),$
\begin{align*}
\big(\boldsymbol{\tau}\cdot\lambda^\ell(\zeta)\big)^*(x_j) &= x_j \circ [\boldsymbol{\tau} \cdot\lambda^{\ell}(\zeta)] = \zeta^{\ell_j}( \boldsymbol{\tau}^*x_j), \\ \big(\boldsymbol{\tau}\cdot\lambda^\ell(\zeta)\big)^*(dx_j)&= dx_j \circ [\boldsymbol{\tau} \cdot \lambda^{\ell}(\zeta)] = \zeta^{\ell_j}( \boldsymbol{\tau}^* dx_j).
\end{align*}
As a consequence, the term $x^\nu$ in $\omega_I$ discharges $\zeta^{\langle \ell, \nu \rangle},$ and $dx_I$ produces $\zeta^{\langle \ell, \mathbbm{1}_I \rangle},$ where $\mathbbm{1}_I$ is the characteristic vector of $I\subseteq [n].$ We conclude
$$
\langle \Phi^*_m [Z]\, , \, \omega_I \rangle =
\sum_{\ell \in \ZZ^n \cap [0,m-1]^n} \zeta^{\langle \ell \, , \, \nu + \mathbbm{1}_I \rangle } \int_{K_{m,1}} \Phi_m^*[Z]\wedge \omega_I.
$$
The final sum, however, vanishes for a fixed $\nu= (\nu_1, \dots , \nu_n)$ and large $m,$ since for $m \not\vert \nu_j,$
$$
\sum_{i=0}^{m-1} (\zeta^i)^{\nu_j}= \frac{(\zeta^m)^{\nu_j}-1}{\zeta^{\nu_j}-1} = 0.
$$

\end{proof}

\section{Applications}\label{sec:apps}
\subsection{From Tori to Toric Varieties}\label{sec:toric_tori}
Let us prove Theorem~\ref{thm:intro_toric} of the introduction.
\begin{theorem}\label{thm:toric_tori}
Let $Z\subseteq (\CC^*)^n$ be an irreducible subvariety of dimension $p,$ and $\bar{Z}$ a tropical compactification of $Z$ in the smooth projective toric variety $X_{\Sigma}$. Then,
$$
\frac{1}{m^{n-p}}\Phi_m^*[\bar{Z}] \longrightarrow \overline{\mathscr{T}}_{\trop(Z)}, \quad \text{as $m\to \infty$},
$$
where $\Phi_m: X_{\Sigma}\longrightarrow X_{\Sigma}$ is the extension of $\Phi_m: (\C^*)^n \longrightarrow (\C^*)^n,$ and $\overline{\mathscr{T}}_{\trop(Z)}$ is the extension by zero of $\mathscr{T}_{\trop(Z)}$ to $X_{\Sigma}.$
\end{theorem}

\begin{proof}
Theorem \ref{thm:main_intro} asserts that
$$
\frac{1}{m^{n-p}}\big(\Phi_m{\upharpoonright}_{(\C^*)^n}\big)^*[Z] \longrightarrow \mathscr{T}_{\trop(Z)}, \quad \text{as $m\to \infty$}.
$$
As in Step 4 of the proof of Theorem~\ref{thm:main_conv},
$$
\frac{1}{m^{n-p}}\overline{\big(\Phi_m{\upharpoonright}_{(\C^*)^n}\big)^*[Z] }\longrightarrow \overline{\mathscr{T}}_{\trop(Z)}, \quad \text{as $m\to \infty$}.
$$
Effectively, we only need to observe the equality 
$$
 \overline{\big(\Phi_m{\upharpoonright}_{(\C^*)^n}\big)^*[Z] } = \Phi_m^*[\bar{Z}],
$$
which is addressed in Remark~\ref{rem:proper_for_all}.
%we probably don't need a tropical compactification.
\end{proof}
% The above theorem is therefore a continuous extension of tropicalisation process from $(\C^*)^n \longrightarrow (\C^*)^n$ to $X_{\Sigma} \longrightarrow X_{\Sigma}$. This extension, therefore, relates to the \emph{Kajiwara--Payne tropicalisation} where one continuously extends the tropicalisation map from $(\C^*)^n \longrightarrow \R^n,$ to
% $$X_{\Sigma }\longrightarrow N_{\Sigma}.$$
% Here, $N_{\Sigma}$ denotes the quotient $X_{\Sigma}\slash S_{N(\{0\})}.$ ref??

\begin{figure}[t]
  \includegraphics[scale = 0.25]{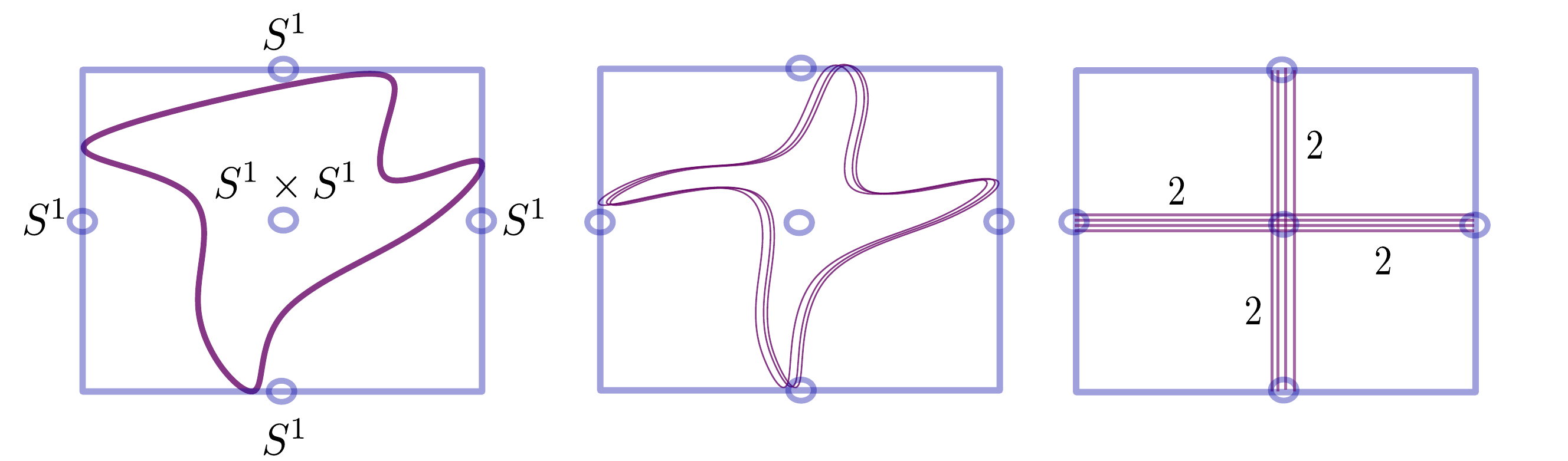}
  \caption{Dynamical tropicalisation in $\P^1 \times \P^1$. See the discussion following the proof of Theorem~\ref{thm:toric_tori}.}
  \label{fig:first_fig}
\end{figure}
Let us now elaborate on Figure~\ref{fig:first_fig}. The toric variety associated to the square is indeed $\P^1 \times \P^1.$ The interior of the square corresponds to the two dimensional orbit $\cO(\{0\}),$ and the sides of the square to one dimensional orbit closures. On the left, the tropical compactification of a curve $C\subseteq (\C^*)^2$ in $\P^1 \times \P^1$ is considered. The closure $\overline{C}$ intersects the one dimensional toric orbits properly and with multiplicity two at all the intersection points. All the toric orbits remain invariant under $\Phi^{-1}_m$, and under this application the points in each toric orbit move towards the compact torus above the corresponding distinguished points. By applying $\Phi^{-1}_m$ to the  curve we obtain a local covering. Multiplying the integration current $\frac{1}{\textrm{deg}(C)}[\Phi_m^{-1}(\overline{C})]$ by $1/m$ normalises the total mass, which is depicted by the thickness of the curves. The limit becomes the associated tropical current whose fibers are transverse to the toric orbits, where the weights account for the intersection multiplicity numbers. Accordingly, the intersection points of multiplicity two converge to the \emph{transverse} (non-normalised) Haar measures on the distinguished circles with mass two.  

\subsection{Dynamical Kapranov Theorem}\label{sec:Kapranov}
Kapranov's theorem~\cite[Theorem 3.1.3]{Maclagan-Sturmfels} is a foundational result in tropical geometry for which we provide a dynamical version in the \emph{trivial valuation} case. 
%Let $K = \C\{\{ t\}\}$ be the field of Puiseux series over $\C$, $\textit{i.e.},$ the elements of the field $K$ are of the form $\sum_{k= k_0}^{\infty} c_{k} t^{k/n},$ for some $k_0\in \Z, c_k\in \C.$ 
% Also assume that $A\subseteq \Z_{\geq 0}^n$ is a finite subset, and $f_t(z) = \sum_{\alpha \in A} c_{\alpha}(t) z^{\alpha} \in K[z]$ is a polynomial such that the coefficients $c_{\alpha}(t)$ are \emph{monomials} in $t,$ \textit{i.e.},
% $f_t(z) = \sum_{\alpha \in A} t^{q_{\alpha}} z^{\alpha},$ for some $q_{\alpha}\in \mathbb{Q}.$ 
Let $A\subseteq \Z_{\geq 0}^n$ be a finite set, and $f(z) =  \sum_{\alpha \in A} c_{\alpha} z^{\alpha} \in \C[z]$ be a complex polynomial. We set the \emph{tropicalisation} of $f$ to be
$$
\trop(f):= \max_{\, \alpha \in A} \{ \langle -\alpha ,~\cdot~ \rangle \}:\R^n \longrightarrow \R.
$$
%  We define,
% $$
% \Phi_m: \C^* \times (\C^*)^n \longrightarrow \C^* \times (\C^*)^n, \quad (t,z) \longmapsto(t^m, \Phi_m(z)).
% $$
 We have
$$
\frac{1}{m}\Phi_m^*\big(\log|f| \big) = \frac{1}{m}\big(\log \big|\sum_{\alpha \in A} c_{\alpha} z^{m\,\alpha} \big| \big).
$$
Similar to the observation in the beginning of Section \ref{sec:trop_algebra} we obtain
\begin{equation}\label{eq:Kapranov_conv}\tag{III}
\frac{1}{m}\Phi_m^*\big(\log|f| \big) \longrightarrow
\trop(f)\circ \Log,
\end{equation}
in $L^1_{\textrm{loc}}(\C^n)$. See also \cite[Theorem 3.4]{Rash}. We note that the negative coefficient of $\alpha$ in $\trop(f)$ is to compensate for using $\Log= -\log \otimes 1$ in the preceding formula.  
\vskip 2mm 
The (dynamical) Kapranov's theorem asserts that the tropicalisation of the algebraic hypersurface $V(f)$ coincides with the (tropical current of the) tropical variety of tropicalisation of $f.$ 
\vskip 2mm
Here is Theorem \ref{thm:intro_Kapranov} of the introduction. 
%
%The dynamical Kapranov theorem, in essence, is the tropicalisation of the Lelong--Poincar\'e equation: 

\begin{theorem}[Dynamical Kapranov Theorem]\label{thm:dyn_Kapranov}
Assume that $[V(f)]= dd^c \log|f|,$ is an integration current in $\sD_{n-1,n-1}'((\C^*)^n)$. Let $\mathfrak{q}:= \trop(f),$ defined as above. We have that,
$$m^{-1} \Phi_m^*[V(f)] \longrightarrow \sT_{V_{\trop}(\mathfrak{q})}.$$ 
\end{theorem}
\begin{proof}
We have $\Phi_m^*[V(f)]= \Phi_m^*(dd^c \log|f|) =dd^c (\Phi_m^* \log|f_t|).$ Dividing by $m$ and taking the limit yields
\begin{equation*}
m^{-1} \Phi_m^*[V(f)] =  dd^c \big[m^{-1}\Phi_m^* \log|f|\big]   \longrightarrow dd^c[\mathfrak{q} \circ \Log] = \sT_{V_{\Trop}(\mathfrak{q})}.
 \end{equation*}
For the above convergence, we have used Equation~(\ref{eq:Kapranov_conv}) and the continuity differentials with respect to the weak limit, and for the latter equality we have used the tropical Lelong--Poincar\'e equation; see Proposition \ref{prop:Lelong-Poincare}.
\end{proof}

\begin{remark}
\begin{itemize}
    % \item [(a)] An interesting related theorem is given in \cite[Proposition 3.1, Theorem 3.4]{Rash}, where it is proved that when
    % $u$ is a plurisubharmonic function of logarithmic growth in $\C^n$, then
    % $$\frac{1}{m}\Phi_m^*(u)\longrightarrow \psi_u \circ \Log.$$
    % Here $\psi_u \circ \Log$ is the \emph{logarithmic indicator} of the function $u$, and $\psi_u:\R^n \longrightarrow \R$ has certain combinatorial convexity properties.
    
    \item [(a)]We have observed that the tropicalisation of the Lelong--Poincar\'e equation leads to a version of Kapranov's theorem. It is therefore tempting to tropicalise the King's formula \cite[III.8.18]{DemaillyBook1}, and wish to obtain a version of the Fundamental Theorem of Tropical Geometry, see \cite[Section 3.2]{Maclagan-Sturmfels}. However, Monge--Amp\`ere operators are not in general weakly continuous, see \cites{Lelong-Discont, Cegrell-discont, Fornaess-Sibony}, and naively applying the  dynamical tropicalisation to the King's formula does not provide  useful information.
    
    \item[(b)] Let  $W \subseteq \C^* \times (\C^*)^n$ be an irreducible subvariety of dimension $p+1$ such that the projection of $W$ onto the first factor is \emph{flat} and surjective. We may therefore regard $W$ as a one-parameter family of algebraic varieties $W_t,$ where each $W_t$ is the $p$-dimensional fiber above $t \in \C^*.$ In tropical geometry, it is shown that the limit in the Hausdorff sense of
$$
\frac{1}{\log|t|}\Log(W_t), \quad \text{as $t\to \infty$,}
 $$
is a rational polyhedral complex which is balanced. See \cites{Rull-thesis, Pass-Rull} for the codimension one case, and \cite{Jonsson} in any codimensions and the background on the evolution of proofs. The correct analogue of this statement in our setting seems to be proving that the sequence 
 $$\frac{1}{m^{n-p+1}}\Phi_m^* \big[(e^m,W_{e^m})\big], \quad \text{as $m\to \infty,$}$$
 converges to a tropical current. See also  \cite[Theorem 5.2.7]{Babaee} for a related result.
    
\end{itemize}
\end{remark}

\subsubsection{Tropical Lelong--Poincar\'e Equation}\label{sec:trop_Lelong-Poincare}
 Chambert-Loir and Ducros in \cite{Chambert-Ducros} proved the generalised Lelong--Poincar\'e equation for Lagerberg currents. We have the following theorem which also appeared in \cite{Babaee}.
\begin{proposition}\label{prop:Lelong-Poincare}
For any tropical polynomial $\mathfrak{q}:\RR^n \to \R$,
$$
dd^c [\mathfrak{q} \circ \Log] = \sT_{V_{\Trop}(\mathfrak{q})},
$$
in $\mathscr{D}_{p,p}'((\C^*)^n)$.
\end{proposition}
\begin{proof}
On the one hand, $\mathfrak{q}$ is a piece-wise linear function, and it is easy to see that the support of $dd^c [\mathfrak{q} \circ \Log]$ coincides with the set $\Log^{-1}(|V_{\trop}(\mathfrak{q})|),$ on the other hand, $dd^c [\mathfrak{q} \circ \Log]$ is $(S^1)^n$ invariant. As a consequence of Proposition~\ref{prop:charac}, $dd^c [\mathfrak{q}\circ \Log]$ is a tropical current. By Demailly's first theorem of support, \cite[Theorem III.2.10]{DemaillyBook1}, we only need to verify the equality of these $(n-1,n-1)$-dimensional currents on $$\bigcup_{\sigma,~ \dim(\sigma)=n-1}\Log^{-1}(\sigma^{\circ}).$$

Therefore, it only remains to see that the weights induced by $dd^c [\mathfrak{q} \, \circ \, \Log]$ on any $(n-1)$-dimensional cone in $V_{\trop}(\mathfrak{q})$ coincides with the weights given in Definition~\ref{def:tropical_poly}.(b). We can assume, without loss of generality, that $\sigma$ contains the origin. To prove the equality of the currents, it is sufficient to prove that for any ball $B$ such that $B \cap \Log^{-1}(\sigma^{\circ})\neq \varnothing,~B \cap \partial \Log^{-1}(\sigma^{\circ}) = \varnothing,$ and $B\cap \Log^{-1}\big( |V_{\Trop}(\mathfrak{q})| -\sigma\big) = \varnothing,$
$$
dd^c [\mathfrak{q} \circ \Log]{\upharpoonright}_B = \sT_{V_{\Trop}(\mathfrak{q})}{\upharpoonright}_B.
$$
Assume that $\text{aff}(\sigma) = \{x\in \R^n: \langle \beta_1 , x \rangle= \langle \beta_2 , x \rangle =\mathfrak{q}(x) \},$ for some $\beta_1,\beta_2\in \Z^n.$ We have that
\begin{align*}
dd^c [\mathfrak{q} \circ \Log]{\upharpoonright}_B &= dd^c \max \{\langle \beta_1 , \Log(z) \rangle , \langle \beta_2 , \Log(z) \rangle \}{\upharpoonright}_B, \\ \sT_{V_{\Trop}(\mathfrak{q})}{\upharpoonright}_B &=w_{\sigma} \sT_{\text{aff}(\sigma)}{\upharpoonright}_B \, .
\end{align*}
Therefore, it suffices to prove that
$$
w_{\sigma} \sT_{\text{aff}(\sigma)}= dd^c \max \{\langle \beta_1 , \Log(z) \rangle , \langle \beta_2 , \Log(z) \rangle \}.
$$
Recall by Definition~\ref{def:tropical_poly}, $w_{\sigma}$ is defined to be the lattice length of $\beta_1 - \beta_2.$ Let us write
$$
\beta_1 - \beta_2 = w_{\sigma} \alpha = w_{\sigma}(\alpha_{+}- \alpha_{-}),
$$
where $\alpha$ is a primitive vector and $\alpha_{\pm}\in \Z^n_{\geq 0}.$ We have
$$
dd^c \max \{\langle \beta_1 , ~.~ \rangle , \langle \beta_2 ,~.~ \rangle \}\circ \Log(z) = w_{\sigma}dd^c \max \{\langle \alpha_{+} ,~.~ \rangle , \langle \alpha_{-} , ~.~ \rangle \}\circ \Log(z).
$$
Finally, to see that the masses of $\sT_{\text{aff}(\sigma)}$ and $dd^c \max \{\langle \alpha_{+} ,~.~ \rangle , \langle \alpha_{-} , ~.~ \rangle \}\circ \Log(z)$ coincide, we note that $\sT_{\text{aff}(\sigma)}$ is an average with respect to the Haar measure of the fibers, and it is enough to show that the mass of each fiber coincides with the mass of $dd^c \max \{\langle \alpha_{+} ,~.~ \rangle , \langle \alpha_{-} , ~.~ \rangle \}\circ \Log(z).$ By Example~\ref{Ex:hypersurface}, for each $x \in (S^1)^n,$ the integration current on each fiber is given by
$$\big[\pi_{\text{aff}(\sigma)}^{-1}(\bar{x})\big]=dd^c \log|(x\cdot z)^{\alpha_{+}} - (x\cdot z)^{\alpha_{-}}|,$$
where $\bar{x} = x + S_{(H \cap \mathbb{Z}^n)}$ is the element in the quotient $S_{\mathbb{Z}^n/(H \cap \mathbb{Z}^n)}$. Applying $\frac{1}{m}\Phi^*_m$ on both sides of this equation preserves the mass in any compactification of $(\C^*)^n$, but by Equation~(\ref{eq:Kapranov_conv}) the left hand side becomes
$dd^c \max \{\langle \alpha_{+} ,~.~ \rangle , \langle \alpha_{-} , ~.~ \rangle \}\circ \Log(z).$
\end{proof}

\subsection{The Genericity Condition in $\P^n$}\label{sec:no-contradiction}
Comparing our convergence results to the conjecture of Dinh and Sibony, we need to verify that there all the \emph{generic} algebraic subvarieties $Z\subseteq \mathbb{P}^n$ of a given degree have the same tropicalisation. This implication, in fact, is the main result of main theorem of R\"{o}mer and Schmitz in  \cite{Romer-Schmitz}. Roberto Gualdi proposed a nice and intuitive explanation of this implication in codimension one: if we consider a homogeneous polynomial of degree $d$ in $\mathbb{P}^n,$ then all the coefficients are generically non-zero, which in turn, imposes the Newton polytope, hence the tropicalisation (with respect to trivial valuation). The genericity condition in dynamical systems usually amounts to avoiding certain invariant sets of low dimension.  In Section~\ref{sec:trop-toric}, we observed that all the toric orbits are invariant under $\Phi_m$ and $\Phi_m^{-1},$ and the relation between the toric orbits and tropicalisation is already stated in the following proposition by Sturmfels and Tevelev.
\begin{proposition}[{\cite[Proposition 3.9]{Sturmfels-Tevelev}}]\label{prop:sturm-tev}
Let $\Sigma$ be a complete (rational) fan in $N_{\R}$ and $Z$ be a $p$-dimensional  subvariety of $(\C^*)^n.$ Assume that the closure $\bar{Z}\subseteq X_{\Sigma},$ does not intersect any of the toric orbits of $X_{\Sigma}$ of codimension greater than $p.$ Then, $\trop(Z)$ equals the union of all $p$-dimensional cones $\sigma\in \Sigma$ such that $\cO_{\sigma}$ intersects $\bar{Z}.$
\end{proposition}

To deduce that the genericity condition for an algebraic subvariety of $\P^n$ determines its tropicalisation, we will  observe that the union of $p$-dimensional cones in $\P^n$ is the support of a \emph{ strongly extremal} fan. Recall that
\begin{definition}
\begin{itemize}
 \item [(a)] A tropical variety $\cC$ is called \emph{strongly extremal} if the support of $\cC$ is can be uniquely weighted up to a multiple to become balanced. In other words, $\cC$ generates an extremal ray in the cone of positively weighted tropical varieties. 
 %Note that for two tropical cycles $\cC_1, \cC_2$ the addition $\cC_1 + \cC_2$ is obtained by a common refinement of $|\cC_1|\cup |\cC_2|$ and addition of their corresponding weights. 

  % \item A tropical cycle is called \emph{strongly extremal} if it is the support of unique tropical cycle up to a multiple.
  \item [(b)] A current $\sT$ in the cone of positive closed currents of bidimension $(p,p)$ on $X$ is called \emph{extremal}, if any decomposition $\sT = \sT_1 +\sT_2$ in this cone implies that $\sT_1$ and $\sT_2$ are positive multiples of $\sT.$ In other words, $\sT$ generates an extremal ray in the cone of positive closed currents of bidimension $(p,p).$
\end{itemize}
\end{definition}

\begin{lemma}\label{lem:str_ex_subset}
Let $\cC, \tilde{\cC}$ be two tropical varieties of same dimension such that  $\cC$ is strongly extremal, and $|\tilde{\cC}|\subseteq |\cC|.$ Then $|\tilde{\cC}|= |{\cC}|.$

\end{lemma}

\begin{proof}
There is a sufficiently large integer $k$, such that the tropical variety $k \, \cC-\tilde{\cC}$ is a (positively weighted) tropical variety with support $|\cC|$. If $|\tilde{\cC}|\subsetneq  |{\cC}|,$ the weights on $k \, \cC-\tilde{\cC}$ cannot be a multiple of the weights on $\cC.$

\end{proof}

\begin{remark}
    It can be shown that if $\cC\subset \R^n$ is strongly extremal, and it does not lie in any proper affine subspace of $\R^n,$ then $\sT_{\cC}$ is extremal; see  \cite{Babaee, BH}. 
\end{remark}
% \begin{proposition}\label{prop:strong-ex}
% If a positively weighted tropical cycle $\cC$ is extremal, then it is also strongly extremal. \end{proposition}

% \begin{proof}
% Assume that $\tilde{\cC}$ is a balanced complex such that $|\cC| = |\widetilde{\cC}|$. For sufficiently large positive integers $k_1, k_2$ both $k_1 \cC + \widetilde{\cC}$ and $k_2 \cC - \widetilde{\cC}$ are positively weighted tropical cycles. Writing $(k_1 \cC + \widetilde{\cC}) + (k_2 \cC - \widetilde{\cC})= (k_1 + k_2) \cC,$ extremality of $\cC$ implies that there exists a positive rational number $q$ such that
% $\cC =q(k_1 \cC + \widetilde{\cC}),$ which implies $\cC = \big(\frac{q}{qk-1}\big)\widetilde{\cC}.$
% \end{proof}

Essentially the same proof as above gives the following. 
\begin{lemma}\label{prop:extremal_subset}
Assume that $\cC, \tilde{\cC}$ are two tropical varieties such that $\sT_{\cC}$ is extremal among the tropical currents, and $|\tilde{\cC}|\subseteq |\cC|.$ Then $|\cC|= |\tilde{\cC}|.$
\end{lemma}
% \begin{proof}
% For a sufficiently large integer $k$ the tropical variety $k \, \cC-\tilde{\cC}$ is positively weighted. We can therefore write $(k\,\cC - \tilde{\cC})+ \tilde{\cC} = k\, \cC.$ By extremality of $\cC$ we have $\cC = s\, \tilde{\cC}$ for a positive rational number $s,$ and the assertion follows.
% \end{proof}

For a given polyhedral complex $\Sigma$, we denote by $\Sigma(p)$ the union of all $p$-dimensional cells in $\Sigma.$ We give two different  proofs of the following:

\begin{theorem}[\cite{Romer-Schmitz}]\label{thm:generic}
Let $\Sigma$ be the fan of $\P^n,$ and $Z$ be any $p$-dimensional subvariety of $(\C^*)^n$ of degree $d$, and generic in the sense that $\bar{Z}\subseteq \P^n$ does not intersect any of toric orbits of $X_{\Sigma}=\P^n$ with codimension greater than $p.$ Then, $\trop(Z)$ is independent of the  choice of $Z$.
\end{theorem}

%It is sufficient to prove that $\bigcup_{\sigma \in \Sigma(k)} \sigma$ is the support of an extremal tropical variety, then the weights are induced by total mass of $Z.$
\begin{proof}
In view of Proposition~\ref{prop:sturm-tev} we have the inclusion $|\Trop(Z)|\subseteq \bigcup_{\sigma \in \Sigma(p)} \sigma.$ By Lemma~\ref{lem:str_ex_subset}, to derive the converse inclusion it is sufficient to observe that the set $\bigcup_{\sigma \in \Sigma(p)} \sigma$ is the support of a strongly extremal tropical variety.  Let us justify the extremality in two ways by devising several strong theorems. See also \cite[Exercise 6.8.11]{Maclagan-Sturmfels}.
\begin{itemize}
  \item [(a)] $\bigcup_{\sigma \in \Sigma(p)} \sigma$ is the support of the \emph{Bergman fan of the uniform matroid}  $U_{p+1, n+1}$; see \cite[Example 4.2.13]{Maclagan-Sturmfels}.  Therefore it is extremal by {\cite[Theorem 38]{Huh}} and the assertion follows from Lemma \ref{lem:str_ex_subset}.
 
 \item [(b)] Let $H\subseteq (\C^*)^n$ be a generic $p$-dimensional plane. Note that this genericity condition for $H$ is equivalent to saying that all the maximal minors of the matrix of coefficients for any $n-p$ linear equations defining $H$ are non-zero. By \cite[Example 4.2.13]{Maclagan-Sturmfels} for such a generic $p$-dimensional plane $H,$ the set $\trop(H)$ is given by support of the Bergman fan of the uniform matroid: $\bigcup_{\sigma \in \Sigma(p)} \sigma$. To see that $\trop(H)$ is extremal, we observe instead that the current $\sT_{\trop(H)}$ is extremal among all $\Phi_m$-invariant currents. By Theorem~\ref{thm:main_intro} and \cite[Proposition 6.1]{Russakovskii--Shiffman}, for a generic $H$
$$
\sT_{\trop(H)} =\lim_{m\to \infty}m^{p-n}(\Phi_m^*[H]) = \lim_{m\to \infty}m^{p-n}(\Phi^*_m{\omega}^{n-p}),
$$
where $\omega$ is the normalised  Fubini--Study form on $\P^n.$ Finally, the bidimension $(p,p)$ \textit{Green current} of $\Phi_m$, given by $\lim_{m\to \infty}m^{p-n}\Phi^*_m(\omega^{n-p})$, is extremal among $\Phi_m$-invariant currents, {\cite[Theorem 1.0.1]{Dinh-Sibony:superpot}}.

\end{itemize}
As a result, all generic closed $p$-dimensional algebraic varieties of $\P^n$, up to the degree, have the same tropicalisation and,

\end{proof}
% The above theorems suggest that one might not expect the extremality results for the Green currents to hold where the union of co-dimension $k$ strata... \textcolor{blue}{fix the dimension}
The support of the Green current of $\Phi_m$ of bidimension $(p,p)$ is called the \emph{Julia set of order $p$} of $\Phi_m,$ denoted by $J_p$. Therefore, $\Log(J_p)$ coincides with the support of $U_{p+1, n+1}$, \textit{i.e.}, the Bergman fan of uniform matroid of rank $p+1$ on $n+1$ points.
%dimension of Bergman fan= rank -1.

%$\lim_{m\to \infty}m^{p-n}\Phi^*_m(\o-------mega^{n-p})$ is the
%\section{}
\subsection*{Acknowledgements}
I'm grateful to Nessim Sibony for his invitation and hospitality in Orsay in July 2018, as well as, accepting my invitation to visit Bristol for two weeks as a distinguished visitor in September 2020. A visit which was postponed to September 2022 due to the pandemic, and sadly did not take place. I'm grateful to Charles Favre for several insightful discussions and his comments on the text. I'm thankful to Omid Amini for numerous encouraging discussions. I also thank Roberto Gualdi for reading the draft, and delightful and fruitful discussions. I cordially thank Tien Cuong Dinh, for his comments on the first arXiv version of this article. I finally thank the anonymous referee(s) for their comments which have led to improvement of mathematics and exposition of the text.

\end{document}